\documentclass[a4paper,11pt]{article}
\usepackage{amsmath,amssymb,amsthm,url,amsfonts}   
\usepackage{makeidx}
\usepackage{mathrsfs}  
\usepackage{mathtools}
\usepackage{bbm}
\usepackage{mathabx}    
\usepackage{enumerate}  
\usepackage[shortlabels]{enumitem}

\allowdisplaybreaks
{\theoremstyle{plain}
	\newtheorem{definition}{Definition}[section]
	\newtheorem{lemma}[definition]{Lemma}
	\newtheorem{theorem}[definition]{Theorem}
	\newtheorem{corollary}[definition]{Corollary}
	\newtheorem{proposition}[definition]{Proposition}
	
} {\theoremstyle{definition}
	\newtheorem{example}[definition]{Example}
	\newtheorem{remark}[definition]{Remark}
}

\renewcommand{\mathbb}{\mathbbm}                     
\renewcommand{\epsilon}{\varepsilon}                 
\renewcommand{\phi}{\varphi}

\renewcommand{\le}{\leqslant}
\renewcommand{\ge}{\geqslant}

\newcommand{\abs}[1]{\left\lvert #1 \right\rvert}    
\newcommand{\norm}[1]{\left\lVert #1 \right\rVert}   

\newcommand{\normm}[1]{{\left\vert\kern-0.25ex\left\vert\kern-0.25ex\left\vert #1 
		\right\vert\kern-0.25ex\right\vert\kern-0.25ex\right\vert}}                      
\newcommand{\normmm}[1]{{\vert\kern-0.25ex\vert\kern-0.25ex\vert #1 
		\vert\kern-0.25ex\vert\kern-0.25ex\vert}}

\DeclareMathOperator{\1}{\mathbbm 1}

\newcommand{\fa}{\qquad\text{for all }}

\DeclareMathOperator{\Borel}{{\mathfrak B}}

\renewcommand{\d}{\mathfrak d}

\newcommand{\F}{{\mathcal F}}

\newcommand{\I}{{\mathcal I}}

\renewcommand{\P}{{\mathcal P}}
\newcommand{\RR}{{\mathcal R}} 
\renewcommand{\S}{{\mathcal S}}

\renewcommand{\O}{{\mathcal O}}

\renewcommand{\I}{{\mathcal I}}

\renewcommand{\d}{{\mathrm d}}

\DeclareMathOperator{\R}{{\mathbb R}}                
\DeclareMathOperator{\Rp}{{\mathbb R}_+}             
\DeclareMathOperator{\N}{{\mathbb N}}                

\DeclareMathOperator*{\esssup}{ess\,sup}

\title{Stochastic integration \\with respect to 
a L\'evy basis}

\makeatletter
\renewcommand{\thefootnote}{\fnsymbol{footnote}} 
\makeatother

\providecommand{\keywords}
{
	\textbf{\textit{Keywords and phrases:}}
}
\providecommand{\amssubj}
{
	\textbf{\textit{AMS 2010 subject classification:}}
}

\author{%
Markus Riedle \\
Department of Mathematics\\ King's College London\\ London WC2R 2LS\\ United Kingdom\\
\texttt{markus.riedle@kcl.ac.uk}
}

\begin{document}
\maketitle

\renewcommand{\thefootnote}{\arabic{footnote}}
\setcounter{footnote}{0}

\begin{abstract}
	We develop a stochastic integration theory for predictable integrands with respect to a L\'evy basis. Our approach is based on decoupling inequalities for tangent sequences and reduces the construction of the stochastic integral essentially to the deterministic integration theory for infinitely divisible random measures developed by Rajput and Rosi\'nski. 
	
	We characterise the corresponding class of integrable predictable processes in terms of the semimartingale characteristics associated with the driving random measure and show that the resulting space of integrands possesses a natural Musielak--Orlicz type structure equipped with an \(F\)-norm. Furthermore, we establish continuity properties of the integral operator and a stochastic version of Lebesgue's dominated convergence theorem.
\end{abstract}

\begin{flushleft}
	\amssubj{Primary 60H05, 60G51; Secondary 60G57, 60H15, 46E30.}
	
\end{flushleft}

\section{Introduction}

Stochastic partial differential equations can be studied either as infinite-di\-men\-sional stochastic evolution equations or within a space-time random field framework. In the latter approach, the driving noise is naturally described by random measures, and a suitable stochastic integration theory with respect to such measures becomes fundamental.

In the Gaussian setting, covariance structures are sufficient to characterise the law of the driving noise and lead naturally to Gaussian random fields and space-time white noise. In the non-Gaussian setting, however, covariance information is no longer sufficient to determine the distribution of the noise, and the random field approach naturally leads to a description of the driving noise in terms of random measures. Consequently, the study of stochastic partial differential equations driven by non-Gaussian noise requires a suitable stochastic integration theory with respect to random measures.

The general theory of stochastic integration with respect to random measures was developed by Bichteler and Jacod \cite{Bichteler-Jacod} following a stochastic Daniell approach to integration; see also the monograph by Bichteler \cite{Bichteler}. This approach is mathematically elegant and conceptually complete, but the resulting definition of the stochastic integral is rather implicit and often technically difficult to apply in concrete situations. This difficulty was addressed by Chong and Kl\"uppelberg \cite{Chong-Kluppelberg}, who, building on the work of Bichteler and Jacod, characterised the space of predictable integrands in terms of semimartingale characteristics associated with the underlying random measure. Their theory has already been successfully applied to stochastic partial differential equations driven by non-Gaussian noise; see, for example, Balan and Jim\'enez \cite{Balan-Jimenez} and Chong, Dalang, and Humeau \cite{Chong-Dalang-Humeau}. 

The purpose of the present article is to develop an alternative approach  for stochastic integration with respect to a L\'evy basis based on explicit topological and functional analytic structures. The central idea of our approach is that stochastic integration for predictable integrands can essentially be reduced to the deterministic integration theory by means of decoupling inequalities for tangent sequences. More precisely, once the stochastic integral is constructed for deterministic integrands, decoupling allows one to separate the randomness of the integrand from that of the integrator and thereby extend the integral to predictable integrands. In this way, the integration theory for predictable processes is largely reduced to the deterministic case. 

In this article,  we develop a stochastic integration theory for predictable integrands with respect to a L\'evy basis and characterise the corresponding largest space of admissible integrable processes in terms of the semimartingale characteristics associated with the driving noise. The resulting space of integrands is shown to possess a natural Musielak--Orlicz type structure equipped with an \(F\)-norm, while the corresponding stochastic integral process is studied in an \'Emery-type topology. In addition, we establish continuity properties of the integral operator together with a stochastic version of Lebesgue's dominated convergence theorem. One advantage of the present approach is that the underlying topological structures are explicit and based on standard functional analytic objects. This makes the stochastic integral technically easier to handle and facilitates the derivation of further structural properties. 

The decoupling approach to stochastic integration originated independently in the 1980s in two different contexts. McConnell \cite{McConnell} developed decoupling methods for stochastic integration in infinite-dimensional spaces, whereas Kwapie\'n and Woyczy\'nski \cite{Kwapien-Woyczynski-1991} used decoupling techniques to characterise admissible integrands for semimartingales. McConnell's approach was later systematically developed by van Neerven, Veraar, and Weis and led to powerful results for stochastic evolution equations driven by cylindrical Brownian motion; see the survey \cite{Neerven-summary}. The second line of development was further extended and presented in detail in the monograph by Kwapie\'n and Woyczy\'nski \cite{Kwapien-Woyczynksi}. 

The article is intended to be largely self-contained and contains detailed arguments throughout. For this reason, we also revisit the deterministic integration theory of Rajput and Rosi\'nski \cite{Rajput-Rosinski}, although our presentation contains several modifications inspired by arguments from \cite{Kwapien-Woyczynksi}. The treatment of predictable integrands follows ideas from our joint work with Bod\'o \cite{Bodo-Riedle-integration}, from which several arguments and aspects of the presentation are adapted. 

The article is organised as follows. In Section \ref{se.inf-divisible}, we briefly introduce independently scattered infinitely divisible random measures and their characteristics following \cite{Rajput-Rosinski}. Section \ref{se.integration-det-RM} develops the stochastic integration theory for deterministic integrands and identifies the corresponding space of integrands as a generalised Musielak--Orlicz space. In Section \ref{se.decoupling}, we introduce decoupled tangent sequences and the corresponding decoupling inequalities. This section is based on \cite{Kwapien-Woyczynksi}, but adapted and streamlined for our purposes. Section \ref{se.integration-predictable} and \ref{se.characterisation-integrands} develop the stochastic integral for predictable integrands and characterise the corresponding space of integrable predictable processes. In the final section, Section \ref{se.limit-theorems}, we establish continuity properties of the integral operator together with a stochastic version of Lebesgue's dominated convergence theorem. 

Our notation is mostly standard. If $(S,\mathcal{S})$ is a measurable space, $\mu$ is a $\sigma$-finite measure on $\mathcal{S}$, and $U$ is a separable topological vector space, then $L_\mu^0(S;U)$, or $L_\mu^0(S,\mathcal{S};U)$ when we wish to emphasise the $\sigma$-algebra on $S$, denotes the space of equivalence classes of measurable functions
$f\colon (S,\mathcal{S})\to (U,\Borel(U))$, where $\Borel(U)$ denotes the Borel $\sigma$-algebra on $U$. A probability space $(\Omega,\F,P)$ is fixed throughout.

\section{Infinitely divisible random measures}\label{se.inf-divisible}

In this section, we introduce random measures with infinitely divisible distributions. These objects are the foundation to define a Lévy-type noise in space and time. They will play a central role in the construction of stochastic models discussed in the next section.

A $\delta$-ring $\RR$ is a family of subsets of a non-empty set $R$ with $\emptyset \in \RR$ and which is closed under finite unions, closed under relative complementation, and closed under countable intersections. We say it includes a partition if there exists
an increasing sequence $(R_k)_{k\in\N}$ of sets $R_k$ in $\RR$ with $R=\bigcup_{k=1}^\infty R_k$. 
Any $\delta$-ring with $R\in \RR$ is a $\sigma$-algebra. 

\begin{definition}\label{de.random-measure-inf}
	Let $\RR$ be a $\delta$-ring with a partition on a non-empty set $R$. 
	\begin{enumerate}
		\item[{\rm (a)}] A  random measure on $\RR$  is a map $\Theta\colon \RR \to L_P^0(\Omega;\R)$ satisfying:
		\[ \Theta\left( \bigcup_{k=1}^\infty A_k\right)= \sum_{k=1}^\infty \Theta(A_k)
		\qquad \text{ $P$-a.s. } \]
		for any sequence $(A_k)_{k\in\N}\subseteq \RR$ of disjoint sets with $\bigcup_{k=1}^\infty A_k\in \RR$.
		\item[{\rm (b)}]  A random measure  $\Theta\colon \RR \to L_P^0(\Omega;\R)$ on $\RR$ is called  independently scattered if 
		$\Theta(A_1), \dots, \Theta(A_n)$ are independent for every disjoint sets $A_1,\dots, A_n\in\RR$.
		\item[{\rm (c)}] An independently scattered random measure   $\Theta\colon \RR \to L_P^0(\Omega;\R)$ is called
      infinitely divisible (resp.\ Poisson) if $\Theta(A)$ is infinitely divisible (resp.\ Poisson) for each $A\in \RR$;
	\end{enumerate}
\end{definition}

In the following, we will typically omit the qualifier independently scattered and simply refer to infinitely divisible random measures, with the understanding that all infinitely divisible random measures considered are independently scattered. 

The infinite divisibility property also guarantees that a random measure admits a L\'evy--Khintchine representation. To state it, we fix throughout this work the truncation function
\begin{align*}
	\tau\colon\R\to\R, \qquad \tau(\alpha)=\begin{cases}
		\alpha, &\text{if }\abs{\alpha}\le 1, \\[0.2em]
		\tfrac{\alpha}{\abs{\alpha}}, &\text{if }\abs{\alpha}>1.
	\end{cases}
\end{align*}
The following L\'evy--Khintchine formula is shown in \cite{Rajput-Rosinski}.
\begin{theorem}\label{th.IDRM-char}
	Let $\Theta\colon \RR \to L_P^0(\Omega;\R)$ be an infinitely divisible random measure. Then, for every $A\in\RR$, the characteristic function of $\Theta(A)$ is given by
	\begin{align*}
		\phi_{\Theta(A)} (\beta)
		=\exp\left(
		i\beta a_0(A)
		-\tfrac12 \beta^2 q_0(A)
		+ \int_{\R} \big(e^{i\beta\alpha} - 1 - i\beta \tau(\alpha)\big)\, \lambda_0(A,\d \alpha)
		\right),
	\end{align*}
	for every $\beta\in\R$, where $a_0\colon \sigma(\RR)\to \R$ is a signed measure,
	$q_0\colon \sigma(\RR)\to [0,\infty)$ is a measure, and
	$\lambda_0\colon \sigma(\RR)\otimes \Borel(\R)\to [0,\infty]$ is a $\sigma$-finite measure.
\end{theorem}

\begin{definition}\label{de.characteristics-RM}
	Let $\Theta\colon \RR \to L_P^0(\Omega;\R)$ be an infinitely divisible random measure
	with characteristic function as  in Theorem \ref{th.IDRM-char}.  We define the characteristics of $\Theta$ to be the triplet $[a_0,q_0,\lambda_0]$ and its control measure  as 
	\[
	\chi\colon \sigma(\RR) \to [0,\infty], \qquad 
	\chi(A)= \abs{a_0}\!(A)+ q_0(A)+ \int_{\R} \big( \abs{\beta}^2\wedge 1\big)\, \lambda_0(A,\d \beta), 
	\] 
	where $\abs{a_0}$ denotes the total variation of the signed measure $a_0$.
\end{definition}

\section{Deterministic integrands}\label{se.integration-det-RM}

Let $\Theta$ be an infinitely divisible random measure on a $\delta$-ring $\RR$ with characteristics $[a_0, q_0, \lambda_0]$ and control measure $\chi$. To develop the integration theory, it is useful to consider the Radon-Nikodym derivatives of the components of the characteristic triplet. Given that $a_0$ and $q_0$ are absolutely continuous with respect to the control measure $\chi$, their densities exist and are defined by
\begin{align} \label{eq.RD-a-q}
	a(r): =\frac{da_0}{d\chi}(r),\qquad
	q(r): =\frac{dq_0}{d\chi}(r).
\end{align}
A disintegration theorem for bimeasures (\cite[Th. 2.4]{Morando} and localising)   guarantees that there exists a kernel $\kappa\colon R\times\Borel(\R)\to [0,\infty]$ satisfying 
\begin{align}\label{eq.RD-lambda}
	\lambda_0(\d r,\d \beta)= \kappa(r,\d\beta)\, \chi(\d r). 
\end{align}

The image of an infinitely divisible random variable under a linear map is again infinitely divisible with known transformation of its characteristics, see \cite[Pro.~11.10]{Sato}.
Thus, for fixed $A\in\RR$ and $\alpha\in\R$, the  random variable $\alpha \Theta(A)$ is infinitely divisible with characteristics $[a_{A,\alpha}, q_{A,\alpha}, \lambda_{A,\alpha}]$ satisfying
\begin{align*}
	&a_{A,\alpha } = \alpha a_0(A)+ \int_{\R} \big( \tau(\alpha\beta)- \alpha\tau(\beta)\big) \lambda_0(A,\d \beta),\qquad  q_{A,\alpha} = \alpha^2 q_0(A)\\
     & \lambda_{A,\alpha}=\lambda_0(A,\cdot)\circ m_\alpha^{-1} \qquad\text{on }\Borel(\R\setminus\{0\}), 
\end{align*}
where $m_\alpha\colon \R\to\R$ is defined by $m_\alpha(\beta)=\alpha \beta$.
Applying  the Radon-Nikodym derivatives \eqref{eq.RD-a-q} and disintegration \eqref{eq.RD-lambda},  and letting
\begin{align}\label{eq.a_Theta}
	a_\Theta(r,\alpha):= \alpha a(r)+ \int_{\R} \big( \tau(\alpha\beta)- \alpha\tau(\beta)\big)\, \kappa(r, \d \beta),
\end{align}
we can rewrite these characteristics as 
\begin{align}\label{eq.RD-multiplied-1}
	&a_{A,\alpha }  =\int_A a_{\Theta}(r, \alpha)\, \chi(\d r),\qquad \quad
	q_{A,\alpha} = \int_A \alpha^2 q(r)\, \chi(\d r)\\
	&\int_{\R} \big(\abs{\beta}^2 \wedge 1\big)\, \lambda_{A,\alpha}(\d \beta)
	= \int_A \int_{\R} \big(\abs{\alpha\beta}^2 \wedge 1\big)\, \kappa(r,\d \beta)\, \chi(\d r). 
	\label{eq.RD-multiplied-2}
\end{align}
Motivated by these representations of the characteristics of $\alpha\Theta(A)$, we define, for $r\in R$ and $\alpha\in\R$, 
\begin{align}\label{eq.multiplied-char}
	\zeta_\Theta(r,\alpha)&= \int_{\R} \left(  \abs{\alpha\beta}^2\wedge 1 \right) \kappa(r,\d \beta) + \alpha^2 q(r),\\
	\eta_\Theta(r,\alpha)&= \sup_{\abs{\beta}\le 1} \abs{ a_\Theta(r,\alpha\beta)}.
\end{align}

\begin{definition}\label{def.RM-det_integrability}
	Let $\Theta\colon \RR \to L_P^0(\Omega;\R)$ be an  infinitely divisible random measure with control measure $\chi$.
	Define the space
		\[
	\I_\Theta:=\left\{ f\in L^0_\chi(R,\sigma(\RR);\R): \iota_\Theta(f)<\infty\right\},
	\]
	where 
	\[
	\iota_\Theta(f):=\int_R \big( \zeta_\Theta(r, \abs{f(r)}) + \eta_\Theta(r,\abs{f(r)})\big) \, \chi(\d r).
	\]

\end{definition}

\begin{example}\label{ex.Poisson-RM}
	Let $\Theta$ be a Poisson random measure.  From the characteristic function of the Poisson distribution, it follows that the characteristics of  $\Theta(A)$ is $[a_0(A), 0,a_0(A)\delta_1]$, where  $a_0(A):=E[\Theta(A)]$ for all $A\in \RR$. Theorem \ref{th.IDRM-char} guarantees that  $a_0$ extends to  a measure on $\sigma(\RR)$. 
	
	The control measure of $\Theta$ is given by $\chi(A)=2a_0(A)$, so that $a(r)=\tfrac12$ for all $r\in R$ and $\kappa( r, \d \beta)=\tfrac12 \delta_1(\d \beta)$. Consequently, we obtain $a_\Theta(r, \alpha)= \tfrac12 \tau(\alpha)$ for each $r\in  R$ and $\alpha \in \R$. Hence, for any measurable function  $f\colon R\to\R$, we obtain
	\[
	\iota_\Theta(f)=\int_{R} \left(\big(\abs{f(r)}^2 \wedge 1\big) + \big(\abs{f(r)}\wedge 1\big)\right)\,  a_0(\d r). 
	\]
	Hence, it follows that $f\in \I_{\Theta}$ if and only if 
	$\displaystyle \int_R \left(\abs{f(r)}\wedge 1\right)\, a_0(\d r)<\infty.$
\end{example}

\begin{example}\label{ex.comp-Poisson-RM}
	Let $\Theta$ be a Poisson random measure with intensity $a_0$ as in Example \ref{ex.Poisson-RM}. Define the compensated Poisson random measure by $\widetilde{\Theta}(A):=\Theta(A)-a_0(A)$ for all $A\in\RR$. Thus, the characteristics of $\widetilde{\Theta}$ is given by $[0,0,a_0(A)\delta_1]$. 
	
	The control measure of $\widetilde{\Theta}$ is given by $\chi(A)=a_0(A)$, so that $a(r)=0$ for all $r\in R$ and $\kappa(r, \d \beta)= \delta_1(\d \beta)$. In this case, we obtain
	  $\eta_{\widetilde{\Theta}}(r,\alpha)=\max\{\abs{\alpha}-1,0\}$. 
	Consequently, it follows that a measurable function  $f\colon R\to\R$ is in $\I_{\widetilde{\Theta}}$ if and only if 
	$\displaystyle \int_R \left(\abs{f(r)}^2\wedge\abs{f(r)}\right)\, a_0(\d r)<\infty.$
\end{example}

\begin{example}
	Let $\rho$ be a $\sigma$-finite measure on $\sigma(\RR)$ which is finite on $\RR$. Let $x,y\in\Rp$ satisfy $x+y=1$, and let $p\in(0,2)\setminus\{1\}$. Define the L\'evy measure
	\[
	\mu_p(\d\beta)
	:=
	\Big(
	xp\beta^{-p-1}\1_{(0,\infty)}(\beta)
	+
	yp(-\beta)^{-p-1}\1_{(-\infty,0)}(\beta)
	\Big)\d\beta.
	\]
	Let $\Theta\colon \RR\to L_P^0(\Omega;\R)$	be an infinitely divisible random measure with characteristics
	\[
	a_0(A)
	=
	\rho(A)\frac{x-y}{1-p},
	\qquad
	q_0\equiv 0,
	\qquad
	\lambda_0(A,\d\beta)
	=
	\rho(A)\mu_p(\d\beta).
	\]
	 The control measure is given by
	$
	A\mapsto
	(
	\frac{|x-y|}{|1-p|}
	+
	\frac{2}{2-p}
	)\rho(A).
	$
	However, we may equivalently choose $\chi:=\rho$. The corresponding Radon-Nikodym derivatives and disintegration kernel are then
	\[
	a(r)=\frac{x-y}{1-p},
	\qquad
	q(r)=0,
	\qquad
	\kappa(r,\d\beta)=\mu_p(\d\beta).
	\]
	A direct calculation shows that for every $\alpha\in\Rp$,
	\begin{align*}
		\zeta_\Theta(r,\alpha)
		&=
		\int_{\R}
		\big(|\alpha\beta|^2\wedge 1\big)
		\,\mu_p(\d\beta)
		=
		\frac{2}{2-p}|\alpha|^p,
		\\
		\eta_\Theta(r,\alpha)
		&=
		\sup_{|\gamma|\le 1}
		|a_\Theta(r,\alpha\gamma)|
		=
		\sup_{|\gamma|\le 1}
		\frac{|x-y|}{|1-p|}
		|\alpha\gamma|^p
		=
		\frac{|x-y|}{|1-p|}|\alpha|^p.
	\end{align*}
	Consequently,
	$
	\I_\Theta
	=
	L_\rho^p(R;\R).
	$
	
	In the case $R=[0,T]\times\R^d$ and $\rho$ the Lebesgue measure on $\Borel([0,T]\times\R^d)$, this example appears in \cite{Balan-stable}.
\end{example}

\begin{theorem}\label{th.RM-space}
	Let $\Theta\colon \RR \to L_P^0(\Omega;\R)$ be an infinitely divisible random measure. Then the space
	$\I_\Theta$	is a complete topological vector space with respect to the topology induced by the $F$-norm
	\[
	\norm{f}_{\Theta}:=\inf\left\{ c>0: \iota_\Theta\!\left(\tfrac{1}{c}f\right)\le c\right\}.
	\]
	Moreover, for every sequence $(f_k)\subseteq \I_\Theta$, we have 
	\[
	\norm{f_k}_\Theta\to 0 \qquad \Longleftrightarrow \qquad \iota_\Theta(f_k)\to 0.
	\]
	Further, step functions are dense in $(\I_\Theta,\norm{\cdot}_\Theta)$
	and the embedding $(\I_\Theta,\norm{\cdot}_\Theta)\hookrightarrow (L_\chi^0(R;\R), d_0)$ is continuous, where $d_0$ metrises the  convergence in measure $\chi$. If the measure space $(R,\sigma(\RR), \chi)$ is separable then $(\I_\Theta,\norm{\cdot}_\Theta)$ is separable. 	
\end{theorem}

\begin{proof}
	Define a function $	\phi\colon R\times [0,\infty)\to [0,\infty)$ by $\phi(r,\alpha):=\zeta_\Theta(r,\alpha)+\eta_\Theta(r,\alpha)$.
	For every $r\in R$, the mapping $\alpha\mapsto \phi(r,\alpha)$ is continuous and non-decreasing on $\R_+$, while for every $\alpha\in\Rp$, the mapping $r\mapsto \phi(r,\alpha)$ is measurable. Hence the functional
    $\iota_\Theta(f)$
	defines a modular in the sense of \cite[Def.~1.1]{Musielak}; see also Example~I in \cite[Section 1.9]{Musielak} and the discussion following Definition~7.1.
	
		By \cite[Th.~1.5]{Musielak}, the functional $\norm{\cdot}_\Theta$ is an $F$-norm on $\widetilde{\I}_\Theta$, where 
	\[
	\widetilde{\I}_\Theta:=\big\{f\in L^0_\chi(R;\R): \lim_{\lambda\to 0}\iota_\Theta(\lambda f)=0\big\}.
	\]
      Since the control measure $\chi$ is $\sigma$-finite, \cite[Th.~7.7]{Musielak} implies that $(\widetilde{\I}_\Theta,\norm{\cdot}_\Theta)$ is complete. 
	
	Next, we verify a $\Delta_2$-condition. Using
	$|\tau(2\beta)-2\tau(\beta)|\le (\tau(\beta))^2$ for every $\beta\in\R$,
	we obtain for every $r\in R$,  $\alpha\in\Rp$, and $\gamma\in [-1,1]$, 
   \begin{align*}
   	\abs{a_\Theta(r,2\alpha\gamma)}
   	&\le 2\abs{a_\Theta(r,\alpha\gamma)}
   	+\int_{\R}\abs{\tau(2\alpha\gamma\beta)-2\tau(\alpha\gamma\beta)}\,\kappa(r,\d\beta)\\
   	&\le 2\abs{a_\Theta(r,\alpha\gamma)}
   	+\int_{\R}(\abs{\alpha\beta}^2\wedge 1)\,\kappa(r,\d\beta). 
   \end{align*}
   Taking the supremum over $\gamma$ gives $\eta_\Theta(r,2\alpha)
   \le 2\eta_\Theta(r,\alpha)+\zeta_\Theta(r,\alpha)$. 
	Since $\zeta_\Theta(r,2\alpha)\le 4\zeta_\Theta(r,\alpha)$, it follows that
	\begin{equation}\label{eq.Mus1-rewrite}
		\phi(r,2\alpha)
		=\zeta_\Theta(r,2\alpha)+\eta_\Theta(r,2\alpha)
		\le 5\zeta_\Theta(r,\alpha)+2\eta_\Theta(r,\alpha)
		\le 5\phi(r,\alpha)
	\end{equation}
	for all $r\in R$ and $\alpha\in\Rp$. Consequently, $\iota_\Theta$ satisfies the 
	$\Delta_2$-condition 
	\begin{align}\label{eq.Delta_2}
		\iota_\Theta(2f)\le 5\,\iota_\Theta(f) \qquad\text{for all } f\in L^0_\chi(R;\R).
	\end{align}
	
	By \cite[Th.\ 7.4]{Musielak}, we obtain $
	\widetilde{\I}_\Theta
	=\big\{f\in L^0_\chi(R;\R): \iota_\Theta(\beta f)<\infty \text{ for some }\beta>0\big\}$.
		Hence $\I_\Theta\subseteq \widetilde{\I}_\Theta$ is immediate. Conversely, let $f\in \widetilde{\I}_\Theta$. Then there exists $\beta>0$ such that $\iota_\Theta(\beta f)<\infty$. Choosing $n\in\N$ such that $2^{-n}\le \beta$, Estimate~\eqref{eq.Delta_2} yields 
 	$
	\iota_\Theta(f)\le 5^n \iota_\Theta(\beta f)<\infty.
	$
	Thus $f\in \I_\Theta$, and so
	$
	\widetilde{\I}_\Theta=\I_\Theta.
	$
	
	Now let $(f_k)\subseteq \I_\Theta$. Since $\iota_\Theta$ satisfies the $\Delta_2$-condition, \cite[Th.~6.2]{Musielak} implies that modular convergence and convergence in the $F$-norm are equivalent, that is 
	$
	\norm{f_k}_\Theta\to 0
	$ if and only if $
	\iota_\Theta(f_k)\to 0.
	$
	Thus $(\I_\Theta,\norm{\cdot}_\Theta)$ is a complete topological vector space.
	
	Finally, we verify the local integrability condition required in \cite[Th.~7.6]{Musielak}. By \eqref{eq.Mus1-rewrite}, for every $\beta>0$ there exists $C_\beta>0$ such that 
	$\phi(r,\beta)\le C_\beta \phi(r,1)$ for all $r\in R$.
	Moreover, for every $A\in \sigma(\mathcal R)$ with $\chi(A)<\infty$, we have
	$\int_A \phi(r,1)\,\chi(\d r)<\infty$.
	Indeed,
	\[
	\eta_\Theta(r,1)
	=\sup_{|\gamma|\le 1}|a_\Theta(r,\gamma)|
	\le |a(r)|+\int_{\mathbb R} (|\beta|^2\wedge 1)\,\kappa(r,\d\beta). 
	\]
	Therefore,
	\begin{align*}
		\int_A \phi(r,1)\,\chi(\d r)
		&\le \int_A |a(r)|\,\chi(dr)+q_0(A)
		+2\int_{\mathbb R} (|\beta|^2\wedge 1)\,\lambda_0(A,\d\beta) \\
		&= |a_0|(A)+q_0(A)
		+2\int_{\mathbb R} (|\beta|^2\wedge 1)\,\lambda_0(A,\d\beta)
		<\infty.
	\end{align*}
	Hence $\phi$ is locally integrable in the sense of \cite[Def.~7.5]{Musielak}. Since \eqref{eq.Delta_2} also implies that
    $\I_\Theta=\widetilde{\I}_\Theta=E^\phi$,
	\cite[Th.~7.6]{Musielak} yields that step functions are dense in $(\I_\Theta,\norm{\cdot}_\Theta)$.
	
	The continuity of the embedding $(\I_\Theta,\norm{\cdot}_\Theta)\hookrightarrow (L^0_\chi,d_0)$
	follows from Musielak, Remark 7.9, applied locally on sets of finite $\chi$-measure, together with  \cite[Th.\ 1.6]{Musielak}. Separability is guaranteed by \cite[Th.\ 7.10]{Musielak}. 
\end{proof}

A simple function on $R$ is of the form
\begin{align}\label{eq.simple-R}
	f\colon R\to \R, \qquad f(r)=\sum_{k=1}^n \beta_k \1_{A_k}(r), 
\end{align}
for some $\beta_k\in\R$ and disjoint sets $A_k\in \RR$ for $k=1,\dots, n$. The space of all simple functions is denoted by $\S_{\rm det}$. The subspace of all simple functions of the form \eqref{eq.simple-R}  with $\abs{\beta_k}\le 1$ for all $k\in \{1,\dots, n\}$ is denoted by $\S_{\rm det}^{1}$. 

The integral of a simple function $f\in \S_{\rm det}$ of the form \eqref{eq.simple-R} with respect to an  infinitely divisible random measure $\Theta$ on $\RR$ is defined as 
\[
\int_{R}f(r)\, \Theta(\d r)= \sum_{k=1}^n \beta_k \Theta(A_k). 
\]
The stochastic integral for deterministic integrands is defined as follows.
\begin{definition}\label{de.integral-det-RM}
	Let $\Theta\colon \RR \to L_P^0(\Omega;\R)$  be an  infinitely divisible random measure with control measure $\chi$.
	A measurable function $f\colon (R,\sigma(\RR))\to (\R,\Borel(\R))$ is $\Theta$-integrable if there 
	exists a sequence  $(f_k)_{k\in\N}$ of simple functions in $\S_{\rm det}$ satisfying
	\begin{enumerate}
		\item[{\rm (a)}] $(f_k)_{k\in\N}$ converges to $f$ $\chi$-almost everywhere;
		\item[{\rm (b)}] $\displaystyle \lim_{k,\ell\to\infty} \sup_{\gamma\in \S_{\rm det}^{1}} E\left[ \abs{\int_R \gamma(r)\big(f_k(r)-f_\ell(r)\big)\, \Theta(\d r)}\wedge 1\right]=0.  $
	\end{enumerate}
	In this case,  the stochastic integral of the deterministic function $f$ is defined by
	\[I_\Theta(f):=\int f \,  \d \Theta:= \lim_{k\rightarrow \infty} \int_R f_k(r) \, \Theta(\d r) \quad \text{in}\; L_P^0(\Omega;\R).\]
\end{definition}
The convergence in Part (b) guarantees that $(\int_{R} (f_k-f_\ell)\,\Theta)_{k,\ell\in\N}$  is a Cauchy sequence $L_P^0(\Omega;\R)$, and thus the limit is well-defined. Standard arguments show that the limit does not depend on the approximating sequence $(f_k)$.

The following result exactly characterises the space of deterministic functions integrable with respect to a given infinitely divisible random measure $\Theta$.
\begin{theorem}\label{th.RM-det_if_and_only_if_integrable}
	Let  $\Theta\colon \RR \to L_P^0(\Omega;\R)$ be an  infinitely divisible random measure. 
	Then the space of  measurable functions $f\colon (R,\sigma(\RR))\to (\R,\Borel(\R))$  integrable with respect to $\Theta$ coincides with the space $\I_\Theta$.
\end{theorem}

We need a few results for the proof of Theorem \ref{th.RM-det_if_and_only_if_integrable}, which we will present in the following.  
\begin{lemma}\label{le.RM-drift-equality}
	Let $\Theta\colon \RR \to L_P^0(\Omega;\R)$  be an  infinitely divisible random measure with control measure $\chi$. For  $a_\Theta$ defined in \eqref{eq.a_Theta}, each $f\in \S_{\rm det}$ satisfies  
	\[
	\int_R \sup_{\abs{\beta}\le 1} a_\Theta(r,\beta f(r))\, \chi(\d r)
	= \sup_{\gamma \in \S_{\rm det}^{1}} \int_{R} a_\Theta(r,\gamma(r)f(r))\, \chi(\d r).  
	\]
\end{lemma}
\begin{proof}
	Since the function $\alpha\mapsto a_\Theta(r,\alpha)$ is continuous for each fixed $r\in R$, we can 	
	define the function
	\begin{align*}
		\pi\colon R\times \R\to [-1,1],\qquad
		\pi(r,\alpha)= \min\{\beta\in [-1,1]: a_\Theta(r,\alpha\beta)= \sup_{\abs{\gamma}\le 1} a_\Theta(r,\alpha\gamma)\}.
	\end{align*}
	To show that $\pi$ is measurable with respect to $\sigma(\RR)\otimes \Borel(\R)$, define, for fixed $x\in [-1,1]$, the function 
	\[
	M_x\colon R\times \R\to \R, \qquad 
	M_x(r,\alpha)= \sup_{-1\le \gamma\le x} a_\Theta(r, \alpha\gamma). 
	\]
	The function $M_x$ is measurable with respect to $\sigma(\RR)\otimes \Borel(\R)$, since $a_\Theta\colon R\times \R\to \R$ is measurable in the first argument and continuous in the second argument. Since, for $x\in [-1,1]$,
	\[
	\{(r,\alpha)\in R\times\R:\, \pi(r,\alpha)\le x\}=\{(r,\alpha)\in R\times \R: M_x(r,\alpha)=\sup_{\abs{\gamma}\le 1} a_\Theta(r,\alpha\gamma)\}, 
	\]
	we obtain that $\pi$ is measurable with respect to $\sigma(\RR)\otimes \Borel(\R)$. Since the definition of $\pi$ shows
	\[
	\int_R \sup_{\abs{\beta}\le 1} a_\Theta(r,\beta f(r))\, \chi(\d r)
	=\int_R a_\Theta(r, \pi(r,f(r)) f(r))\, \chi(\d r), 
	\]
	the proof is completed by approximating the measurable function $r\mapsto \pi(r,f(r))$ by simple functions in $\S_{\rm det}^{1}$. 
\end{proof}

\begin{lemma}\label{le.limit_characteristics}
	Let $L$ be a real-valued L\'evy process with characteristics $(b, q, \lambda)$ with respect to the truncation function $\tau$, and let $(\pi_n)_{n \in \mathbb{N}}$ be a normal sequence of partitions $\pi_n=(p_{k,n})_{k=0,\dots, m_n}$ of $[s,t]$. If we put $d_{k,n}=L(p_{k,n})-L(p_{{k-1},n})$, 
	then we have
	\begin{enumerate}\label{le.char_as_limit_of_series}
		\item [\rm (a)] $ \displaystyle \lim_{n \rightarrow \infty} \sum_{k=1}^{m_n} E\left[\tau(d_{k,n})\right]=(t-s)b$;
		\item [\rm (b)] $\displaystyle \lim_{n \rightarrow \infty} \sum_{k=1}^{m_n} E\left[ \tau^2(d_{k,n})  \right]= (t-s)\left(\int_{\R} \left( \abs{\beta}^2 \wedge 1 \right)\, \lambda({\rm d}\beta)+ q\right)$.
	\end{enumerate}
\end{lemma}
\begin{proof}
	Define $e(r):=E[\tau(L(r))]$ for $r\ge 0$. 	
	 Let $A$ denote the generator of $L$. For every $f\colon\R\to\R$ continuous, bounded, and locally twice differentiable in $x\in\R$, we have
	 \[
	 Af(x)
	 =
	 bf'(x)+\frac{q}{2}f''(x)
	 +\int_{\mathbb{R}}\Big(f(x+\beta)-f(x)-\tau(\beta)f'(x)\Big)\,\lambda({\rm d}\beta).
	 \]	
   Applying $A$ to $\tau$, we conclude
    \[
    	\lim_{r\downarrow 0}\frac{1}{r}E[\tau(L(r))]=A\tau(0)=b.
    \]
    Thus, we obtain $e(r)=br+r\epsilon(r)$ for a function $\epsilon\colon\Rp\to\R$ with $\epsilon(r)\to 0$.
     By stationarity of increments, for $\delta_{i,n}:=p_{i,n}-p_{i-1,n}$, we obtain 
    \[
    \sum_{\pi_n}E[\tau(d_{i,n})]=
    \sum_{\pi_n} e(\delta_{i,n})
    =
    b\sum_{\pi_n}\delta_{i,n}
    +
    \sum_{\pi_n}\delta_{i,n}\varepsilon(\delta_{i,n})
    \longrightarrow (t-s)b,
    \]
    which completes the proof of Part (a). 
    
    Part (b) can be proved analogously by considering the map $r\mapsto E[\tau^2(L(r))]$. 
\end{proof}

\begin{lemma}\label{le.independent_sum_control}
		Let $X_1,\dots,X_n$ be independent real-valued random variables. Then
	\[E\left[\abs{\sum_{k=1}^n X_k} \wedge 1 \right]
	\le\left( \sum_{k=1}^n E\left[\tau^2(X_k)\right]  + \left( \sum_{k=1}^n E[{\tau}(X_k)]\right)^2 \right)^{1/2} \!\!\!\!\!
	+ \sum_{k=1}^n E\left[{\tau}^2(X_k)\right] .\]
\end{lemma}

\begin{proof}
	Independence of the random variables implies	
	\[
	E[\tau(X_k)\tau(X_\ell)] =
	E[\tau(X_k)]E[\tau(X_\ell)] 
	\fa k\neq \ell. 
	\]	
	It follows that 
	\begin{align*}
		E\left[ \abs{\sum_{k=1}^n \tau(X_k)}^2\right]
		&= \sum_{k=1}^n E\left[\tau^2(X_k)\right] 
		+ \sum_{k\neq \ell } E[\tau(X_k)]E[\tau(X_\ell)]\\
		&\le \sum_{k=1}^n E\left[\tau^2(X_k)\right]  + \left( \sum_{k=1}^n E[\tau(X_k)]\right)^2.  
	\end{align*}
	Applying this inequality together with the Cauchy--Schwarz and Markov inequalities, we obtain
	\begin{align*}
		E\left[\abs{\sum_{k=1}^n  X_k}\wedge 1 \right]
		&\le E\left[\abs{\sum_{k=1}^n  {\tau}(X_k)}\wedge 1  \right] + P\left( \sum_{k=1}^n X_k \neq \sum_{k=1}^n {\tau}(X_k)\right)\\
		&\le \left(E\left[\abs{\sum_{k=1}^n  {\tau}(X_k)}^2  \right]\right)^{1/2} + P\left( \bigcup_{k=1}^n \{\abs{X_k} > 1\}\right)\\
		&\le\left( \sum_{k=1}^n E\left[\tau^2(X_k)\right]  + \left( \sum_{k=1}^n E[{\tau}(X_k)]\right)^2 \right)^{1/2} 		+ \sum_{k=1}^n E\left[\tau^2(X_k)\right], 
	\end{align*}
	which completes the proof.
\end{proof}

\begin{proposition}\label{pro.RM-cont_of_int_op}
	Let $\Theta\colon \RR \to L_P^0(\Omega;\R)$ be an infinitely divisible random measure. 
	The following are equivalent for a sequence  $(f_k)_{k\in\N}$  in $\mathcal{S}_{\rm det}$: 
	\begin{enumerate}
		\item[{\rm (a)}] $\displaystyle (f_k)_{k\in\N}$ converges to $0$ in $(\I_\Theta,\norm{\cdot}_\Theta)$;
		\item[{\rm (b)}] $\displaystyle \lim_{k\to\infty} \sup_{\gamma \in \mathcal{S}_{\rm det}^1}E\left[\abs{\int_R \gamma f_k \, \d \Theta}\wedge 1 \right]=0. $
	\end{enumerate}
\end{proposition}
\begin{proof} 
	Assume that $[a_0,q_0,\lambda_0]$ is the characteristics of $\Theta$ with the Radon-Niko\-dym derivatives $a$, $q$ and $\kappa$ according to \eqref{eq.RD-a-q} and \eqref{eq.RD-lambda}.
	
	(a)$\;\Rightarrow\;$(b): 
	for any  $\gamma \in \mathcal{S}^{\rm 1}_{\rm det}$ and $f\in \mathcal{S}_{\rm det}$, we can  assume that the pointwise product	$\gamma f$ is of the form 
	\begin{align}\label{eq.gamma-f}
		\gamma(r) f(r)=\sum_{k=1}^{n}\gamma_k \beta_{k}\1_{A_k}(r) \fa r\in R,
	\end{align}
	where $\gamma_k\in [-1,1]$, $\beta_{k} \in \R$ and $A_{k}$ are disjoint sets in $\RR$ for $k=1,\dots, n$.
	For each $k \in \{1,\dots, n \}$, define an infinitely divisible real random variable by $Y_{k}:=\gamma_{k}\beta_{k}\Theta(A_k)$. It follows from \eqref{eq.multiplied-char}, that its characteristics $(a_k,q_k,\lambda_k)$ satisfies
	\begin{align*}
		& a_k= 
		\int_{A_k} a_\Theta(r,\gamma_k\beta_k)\,\chi(\d r), \qquad 
		q_k=\int_{A_k}\gamma_k^2\beta_k^2 q(r)\, \chi(\d r), \\
		& \int_{\R} \left( \abs{\beta}^2\wedge 1\right)\, \lambda_k(\d \beta)
		= \int_{A_k} \int_{\R} \left( \abs{\gamma_k\beta_k \beta}^2\wedge 1\right)\, \kappa(r, \d \beta)\, \chi(\d r). 
	\end{align*}
	Choose, for each $k=1,\dots, n$ and $m\in\N$, a sequence $X_m^k:=(X^{k}_{m,\ell})_{\ell=1,\dots, 2^m}$ of independent, identically distributed random variables $X^{k}_{m,\ell}$ such that their sum has the same distribution as $Y_k$; in other words, $X_{m,1}^k$ is the $2^m$-th root of $Y_k$. 
	It follows from Lemma \ref{le.limit_characteristics} that 
	\begin{align*}
		\lim_{m\to\infty} \abs{\sum_{k=1}^n \sum_{\ell=1}^{2^m} E\left[\tau(X^{k}_{m,\ell})\right] }
		&=\abs{\sum_{k=1}^n a_k }
		=\abs{\int_{R} a_\Theta(r, \gamma(r)f(r))\, \chi(\d r)}\le \iota_\Theta(f),\\
		\intertext{and, in a similar way, }
		\lim_{m\to\infty} \sum_{k=1}^n\sum_{\ell=1}^{2^m} E\left[\tau^2(X^{k}_{m,\ell})\right]
		&=\int_{R} \zeta_{\Theta} (r, \gamma(r)f(r))\, \chi(\d r)\le \iota_\Theta(f). 
	\end{align*}
	Since for fixed $m$,  the sequences $X_m^1,\dots, X_m^n$ can be chosen independently, 
	we conclude by applying Lemma~\ref{le.independent_sum_control} for fixed $m$ and then taking the limit in $m$ that 
	\begin{align*}
		\sup_{\gamma\in \S^{\rm 1}_{\rm det}} E\left[\abs{\int_R \gamma f\, \d \Theta}\wedge 1 \right]
		&=\sup_{\gamma\in \S^{\rm 1}_{\rm det}} E\left[\abs{\sum_{k=1}^n \sum_{\ell=1}^{2^m} X_{m,\ell}^{k} }\wedge 1 \right] \\
		&\le \left( \iota_\Theta(f)+ (\iota_{\Theta}(f))^2\right)^{1/2} + \iota_\Theta(f),
	\end{align*}
	which completes the proof of this implication by Theorem \ref{th.RM-space}.
	
	(b)$\;\Rightarrow\;$(a): 
	for $\gamma\in \S_{\rm det}^{1}$ and $f\in \S_{\rm det}$ let $\gamma f$ be of the form \eqref{eq.gamma-f} and 
	$Y_{\ell}:=\gamma_{\ell}\beta_{\ell}\Theta(A_\ell)$. Define the  integral
	\[ I_\Theta(\gamma f):=\int_R \gamma f\, \d\Theta=\sum_{\ell=1}^{n} Y_\ell, \]
	and denote its characteristics by $(a_I, q_I, \lambda_I)$. Since $I_\Theta(\gamma f)$ is the sum of the independent infinitely divisible random variables $Y_\ell$, its characteristics satisfies 
	\begin{align*}
		\int_{\R} \left(\abs{\beta}^2\wedge 1\right)\, \lambda_I(\d \beta) + q_I
		&= \int_R \zeta_\Theta (r,\gamma(r)f(r))\, \chi(\d r) ,\\
		a_I&= \int_R a_\Theta(r, \gamma(r)f(r))\, \chi(\d r). 
	\end{align*}
	Thus, if the hypothesis under (b) is satisfied by a sequence $(f_k)_{k\in\N}\subseteq \S_{\rm det}$,  
	Remark VII.2.10 in \cite{JacodShiryaev} guarantees by taking $\gamma\equiv 1$ on the support of $f_k$ that 
	\begin{align} 
	\lim_{k\to\infty}\int_R \zeta_\Theta(r, \abs{f_k(r)})\,\chi(\d r)=	\lim_{k\to\infty}\int_R \zeta_\Theta(r, f_k(r))\,\chi(\d r)= 0.\label{eq.RM-aux1}
	\end{align}
	In order to establish 
	\begin{align}\label{eq.RM-aux3}
		\lim_{k\to\infty}\sup_{\gamma\in \S_{\rm det}^{1}}\int_R a_\Theta(r, \gamma(r) f_k(r))\, \chi(\d r)=0,
	\end{align}
	assume for a contradiction, by passing to a suitable subsequence if necessary, that  there exist an $\epsilon>0$ and a sequence $(\gamma_k)_{k\in\N}\subseteq \S_{\rm det}^{1}$ satisfying for all $k\in\N$ that
	\[
	\int_R a_\Theta(r,\gamma_k(r)f_k(r))\, \chi(\d r)>\epsilon. 
	\] 
	Since the hypothesis implies that $I_\Theta(\gamma_k f_k)\to 0$ weakly,  the characterisation of weak convergence of infinitely divisible measures, see e.g.\ \cite[Th.\ VII.2.9]{JacodShiryaev}, guarantees that 
	\[
	\lim_{k\to\infty} \int_R a_\Theta(r, \gamma_k(r)f_k(r))\, \chi(\d r)=0.
	\]
	Thus, we reach a contradiction, which proves \eqref{eq.RM-aux3}. Since $\beta\mapsto a_\Theta(r,\beta)$ is an odd function for each $r\in R$, we conclude from \eqref{eq.RM-aux3} by  Lemma \ref{le.RM-drift-equality}  that
	\begin{align}
	\lim_{k\to\infty}\int_R \eta_\Theta(r, \abs{f_k(r)})\, \chi(\d r)=	\lim_{k\to\infty}\sup_{\gamma\in \S_{\rm det}^{1}}\int_R a_\Theta(r, \gamma(r) f_k(r))\, \chi(\d r)=0.
	\end{align}
	This together with   \eqref{eq.RM-aux1} establishes $\iota_\Theta(f_k)\to 0$.
\end{proof}

\begin{proof}[Proof of Theorem \ref{th.RM-det_if_and_only_if_integrable}]
	If $f\colon R\to\R$ is $\Theta$-integrable, then Definition \ref{de.integral-det-RM} guarantees that there exists a sequence $(f_k)_{k \in \mathbb{N}}$ of simple functions in $ \mathcal{S}_{\rm det}^{}$ such that $f_k \rightarrow f$ $\chi$-almost everywhere and  $\sup_{\gamma \in \mathcal{S}_{\rm det}^{1}}E[\abs{I_\Theta(\gamma(f_k-f_\ell))}\wedge 1] \rightarrow 0$. Proposition \ref{pro.RM-cont_of_int_op} implies that $f_k-f_\ell \rightarrow 0$ in $(\I_\Theta,\norm{\cdot}_\Theta)$. Completeness of the space $\I_\Theta$ and its continuous embedding into $L_\chi^0(R;\R)$  allow us to conclude that $f \in \I_\Theta$.
	
	Conversely, if  $f \in \I_\Theta $, then we conclude from Theorem \ref{th.RM-space}  that there exists a sequence $(f_k)_{k \in \mathbb{N}}$ of elements in $\mathcal{S}_{\rm det}$ converging to $f$ $\chi$-almost everywhere and in $(\I_\Theta,\norm{\cdot}_\Theta)$. Applying Proposition  \ref{pro.RM-cont_of_int_op} completes the proof.
\end{proof}

\section{Tangent sequences and decoupling inequalities} \label{se.decoupling}

The technique of constructing decoupled tangent sequences is a powerful tool to derive strong results on a sequence of possibly dependent random variables. In this section, we briefly recall the fundamental definition and result, see e.g.\   de la Pe\~{n}a and Gin\'{e} \cite{Pena-Gine} and Kwapie{\'n} and Woyczy{\'n}ski \cite{Kwapien-Woyczynksi}.

\begin{definition}\label{de.decoupled}
	Let $\big(\Omega, \mathcal{F}\!,P,(\mathcal{F}_k)_{k\in{\mathbb N}_0}\big)$ be
	a filtered probability space and $(X_k)_{k \in \mathbb{N}}$ an adapted sequence of real-valued random variables. If $\big(\Omega', \mathcal{F}',P',(\mathcal{F}'_k)_{k\in{\mathbb N}_0}\big)$ is another filtered probability space, then a sequence $(Y_k)_{k \in \mathbb{N}}$ of real-valued random variables defined on $\big(\Omega \times \Omega^\prime, \mathcal{F}\otimes \mathcal{F}', P \otimes P', (\mathcal{F}_k\otimes \mathcal{F}'_k)_{k\in{\mathbb N}_0}\big)$ is said to be a decoupled tangent sequence to $(X_k)_{k \in \mathbb{N}}$ if
	\begin{enumerate}
		\item[{\rm (a)}]     $(Y_k(\omega,\cdot))_{k \in \mathbb{N}}$ is a sequence of independent random variables on $(\Omega',\mathcal{F}',P')$ for each $\omega \in \Omega$;
		\item[{\rm (b)}]  the sequences $(X_k)_{k \in \mathbb{N}}$ and $(Y_k)_{k \in \mathbb{N}}$ satisfy for each $k\in {\mathbb N}$ that
		\[\mathcal{L}(X_k \vert \mathcal{F}_{k-1}\otimes \mathcal{F}'_{k-1})=\mathcal{L}(Y_k \vert \mathcal{F}_{k-1}\otimes \mathcal{F}'_{k-1}) \quad P \otimes P'\!-\!\text{almost surely. }\]
	\end{enumerate}
\end{definition}

In Definition \ref{de.decoupled} and throughout, we identify a random variable $X\colon \Omega \to \R$ with its canonical extension to the product space $\Omega\times \Omega^\prime$, defined by $X(\omega,\omega') := X(\omega)$ for all $\omega\in \Omega$ and $\omega^\prime\in\Omega^\prime$.

A typical example arising from integration with respect to a process with independent increments is the following:
\begin{example}\label{ex.decoupling-independent}
		Let $(\xi_k)_{k\in\N}$ be a sequence of independent, real-valued random variables and set $\F_k:=\sigma(\xi_1,\dots, \xi_k)$. 
	Let $(\xi_k^\prime)_{k\in\N}$ be an independent copy of $(\xi_k)_{k\in\N}$ defined on a second probability space $(\Omega',\F',P')$ and set $\F_k':=\sigma(\xi_1',\dots,\xi_k')$.
	Let $(\psi_k)_{k\in\N}$ be a sequence of $\F_{k-1}$-measurable, real-valued random variables and set $X_k:=\psi_k\xi_k$. 
	Consider the product space as in Definition \ref{de.decoupled} and define $Y_k:=\psi_k\xi_k^\prime$. 
	Then $(Y_k)_{k\in\N}$ is a decoupled tangent sequence to $(X_k)_{k\in\N}$. Indeed, for each fixed $\omega\in\Omega$, the sequence $(Y_k(\omega,\cdot))_{k\in\N}$ is independent on $(\Omega',\F',P')$. Moreover, for every bounded Borel measurable function $f\colon \R\to\R$, since $\psi_k$ is $\F_{k-1}\otimes \F_{k-1}'$-measurable and $\xi_k$ is independent of $\F_{k-1}\otimes \F_{k-1}'$, we have
	\[
	E\big[f(X_k)\mid \F_{k-1}\otimes \F_{k-1}'\big]
	=
	E\big[f(\psi_k\xi_k)\mid \F_{k-1}\otimes \F_{k-1}'\big]
	=
	\int_{\R} f(\psi_k\beta)\,\mathcal{L}(\xi_k)(\d\beta),
	\]
	and similarly
	\[
	E\big[f(Y_k)\mid \F_{k-1}\otimes \F_{k-1}'\big]
	=
	E\big[f(\psi_k\xi_k')\mid \F_{k-1}\otimes \F_{k-1}'\big]
	=
	\int_{\R} f(\psi_k\beta)\,\mathcal{L}(\xi_k')(\d\beta).
	\]
	Since $\xi_k'$ is an independent copy of $\xi_k$, we have $\mathcal{L}(\xi_k')=\mathcal{L}(\xi_k)$, which verifies Condition~(b) in Definition \ref{de.decoupled}.
\end{example}
In the following theorem and subsequently, we write $E_P$ and $E_{P\otimes P'}$ to indicate the underlying probability measure with respect to which the expectation is taken.
\begin{theorem}\label{th.decoupling}
	Let $\big(\Omega, \mathcal{F},P,(\mathcal{F}_k)_{k\in{\mathbb N}_0}\big)$ and $\big(\Omega', \mathcal{F}',P',(\mathcal{F}'_k)_{k\in{\mathbb N}_0}\big)$ be filtered probability spaces. 
	Then there exist constants $c_1,c_2>0$ such that, for every $n\in\N$, every adapted sequence 
	$(X_k)_{k\in\N}$ of real random variables on $\big(\Omega, \mathcal{F},P,(\mathcal{F}_k)_{k\in{\mathbb N}_0}\big)$ with corresponding decoupled tangent sequence $(Y_k)_{k\in\N}$ on $\big(\Omega \times \Omega^\prime, \mathcal{F}\otimes \mathcal{F}', P \otimes P', (\mathcal{F}_k\otimes \mathcal{F}'_k)_{k\in{\mathbb N}_0}\big)$ satisfy:
	\begin{enumerate}
		\item[{\rm (a)}]   $\displaystyle
		E_P\left[\left|\sum_{k=1}^n X_k\right|\wedge 1\right]\leq c_1 E_{P\otimes P^\prime}\left[\left|\sum_{k=1}^n Y_k\right|\wedge 1\right]$;
		\item[{\rm (b)}] $\displaystyle
		E_{P\otimes P^\prime}\left[\left|\sum_{k=1}^n Y_k\right|\wedge 1\right]\leq c_2 \sup_{\epsilon_1,\dots,\epsilon_n \in \{\pm 1\}}E_P\left[\left|\sum_{k=1}^n \epsilon_k X_k\right|\wedge 1\right]. $
	\end{enumerate}
\end{theorem}
\begin{proof}
	See \cite[Prop.\ 5.7.1.(ii)]{Kwapien-Woyczynksi} and \cite[Prop.\ 5.7.2]{Kwapien-Woyczynksi}.
\end{proof}

\section{Integral for predictable integrands}\label{se.integration-predictable}

In this section, we introduce a stochastic integral for random integrands with respect to an infinitely divisible random measure. To this end, it is necessary to distinguish a temporal component and to work within a filtration-based framework.
Accordingly, we endow the probability space $(\Omega,\mathcal F,P)$ with a filtration $\{\F_t\}_{t\ge 0}$. 
Throughout this section, let $(\Omega',\mathcal F',P')$ be a copy of $(\Omega,\mathcal F,P)$, and let $(\mathcal F'_t)_{t\ge0}$ denote the corresponding filtration. Product spaces are formed with respect to these copies.

Let $\P$ denote the predictable $\sigma$-algebra on $\Omega\times [0,T]$ for some $T>0$, and let $\O\subseteq \R^d$ be a Borel set. A stochastic process
$\Xi\colon \Omega\times [0,T]\times \O\to \R$ 
is called predictable if it is measurable with respect to $\P\otimes \Borel(\O)$.

For $D\in\Borel(\R^n)$, let $\Borel_b(D)$ denote the $\delta$-ring on $D$ of all sets $A\in\Borel(\R^n)$ such that $A\cap D$ has finite Lebesgue measure. 
We extend the notion of an infinitely divisible random measure to include a temporal domain as follows:
\begin{definition} \label{de.Levy-space-time}
	Let $\O$ be in $\Borel(\R^d)$. 
	An infinitely divisible random measure $\Theta\colon \Borel_b(\Rp\times \O) \to L_P^0(\Omega;\R)$ is called a L\'evy basis if: 
	\begin{enumerate}
		\item[{\rm (a)}] for every $A\in\Borel_b(\Rp\times \O)$ for which there exists $t\ge 0$ such that $A\subseteq [0,t]\times \O$, 
		the random variable $\Theta(A)$ is $\F_t$-measurable; 
		\item[{\rm (b)}] for every $A\in\Borel_b(\Rp\times \O)$ for which there exists $t\ge 0$ such that $A\subseteq (t,\infty)\times \O$,  
		the random variable $\Theta(A)$ is independent of $\F_t$;
		\item[{\rm (c)}] $\Theta(\{t\}\times [-n,n]^d\cap\O)=0$ $P$-a.s.\ for all $t\ge 0$ and $n\in\N$.
	\end{enumerate}
\end{definition}

Conditions (a) and (b) in Definition \ref{de.Levy-space-time} correspond to the usual conditions of adaptedness and independent increments. Condition (c) ensures that the noise has no fixed times of discontinuity.

Throughout, we fix a finite time horizon $T>0$. Although the filtration and the Lévy basis are defined on $[0,\infty)$, all predictable integrands and stochastic integrals are considered on $[0,T]$.
 As in the case of deterministic integrands, we begin by introducing two classes of functions on which our definition of the stochastic integral will be based.
\begin{definition} \hfill 
	\begin{enumerate}
		\item[{\rm (a)}] A real-valued predictable step process $\Xi \colon \Omega \times [0,T]\times \O\to \R$ is said to be of the form 
		\begin{align} \label{eq.RM-step-HS}
			\Xi(\omega, t,r)=\sum_{k=0}^{n-1} \xi_k(\omega)  \mathbb{1}_{ (t_k,t_{k+1}]}(t)\mathbb{1}_{A_k}(r),
		\end{align}
		where $0=t_0<\cdots < t_{n}=T$, $\xi_k$ is an $\F_{t_k}$-measurable real-valued random variable, and 
		$A_k\in \Borel_b(\O)$ for $k=0,\dots, n-1$. The space of all real-valued predictable step processes is denoted by $\mathcal{S}_{\rm prd}$. 
		\item[{\rm (b)}] The subspace of all real-valued predictable step processes $\Gamma$ satisfying 
		\begin{align*}
		\esssup_{\omega\in\Omega}	\sup_{(t,r)\in  [0,T]\times \O}\abs{\Gamma(\omega,t,r)}\leq 1
		\end{align*}
		is denoted by  $\mathcal{S}_{{\rm prd}}^{1}$. 
	\end{enumerate}
\end{definition}

 Let $\Xi \in \mathcal{S}_{\rm prd}$ be of the form \eqref{eq.RM-step-HS}. 
 The stochastic integral of $\Xi$ with respect to the L\'evy basis $\Theta$ is defined by
 the real-valued random variable 
 \begin{align}\label{eq.def-int}
 I(\Xi):=\int_{(0,T]\times \O} \Xi \, \d\Theta 
 :=\sum_{k=0}^{n-1} \xi_k \Theta((t_k,t_{k+1}]\times A_k).
 \end{align}
 This definition is independent of the chosen representation by finite additivity of $\Theta$ and the independence assumptions.
  \begin{definition} \label{de.RM-pred_integrability}
	Let $\Theta\colon \Borel_b(\Rp\times \O) \to L_P^0(\Omega;\R)$ be a L\'evy basis with control measure $\chi$. 
	A predictable process $\Xi \colon \Omega \times [0,T]\times \O\to \R$ is called integrable with respect to $\Theta$ if there exists a sequence $(\Xi_k)_{k \in \mathbb{N}}$ of processes in $\mathcal{S}_{\rm prd}$ such that
	\begin{enumerate}
		\item[{\rm (a)}] $(\Xi_k)_{k \in \mathbb{N}}$ converges $P\otimes \chi$-almost everywhere to $\Xi$;
		\item[{\rm (b)}]  $\displaystyle \lim_{k,\ell \rightarrow \infty}\sup_{\Gamma \in \mathcal{S}_{{\rm prd}}^{1}}E\Bigg[\abs{\int_{(0,T]\times \O} \Gamma(\Xi_k-\Xi_\ell) \;{\rm d}\Theta}\wedge1  \Bigg]=0.$
	\end{enumerate}
 	In this case,  the stochastic integral of $\Xi$ is defined by
\[I_\Theta(\Xi):=\int_{(0,T]\times \O} \Xi \;{\rm d}\Theta= \lim_{k\rightarrow \infty} \int_{(0,T]\times \O} \Xi_k \;{\rm d}\Theta \quad \text{in}\;L_P^0(\Omega;\R).\]
\end{definition}

By choosing $\Gamma \equiv 1$ on the support of $\Xi_k-\Xi_\ell$ in Part (b) of Definition \ref{de.RM-pred_integrability}, we obtain that $\big(\int \Xi_k\,{\rm d}\Theta\big)_{k\in\N}$ is a Cauchy sequence in probability, and hence converges in $L_P^0(\Omega;\R)$. Moreover, it can be shown by standard arguments that the limit does not depend on the choice of the approximating sequence $(\Xi_k)_{k\in\N}$.

For an integrable process $\Xi \colon \Omega \times [0,T]\times \O\to \R$, the stochastic integral process $\big(\int_{(0,t]\times\O} \Xi\, \d \Theta:\, t\in [0,T]\big)$ is defined by
\[
\int_{(0,t]\times\O} \Xi \, \d \Theta:=\int_{(0,T]\times \O} \1_{(0,t]}\Xi\, \d \Theta. 
\]
It follows immediately from Definition \ref{de.RM-pred_integrability} that the integral process is well defined. Moreover, Lemma \ref{le.Emery-ucp} below implies that the integral process has c\`adl\`ag paths.

\begin{lemma}\label{le.Emery-ucp}
	Let $\Theta\colon \Borel_b(\Rp\times \O) \to L_P^0(\Omega;\R)$ be a L\'evy basis and  $\Xi\in \mathcal S_{\rm prd}$. Then, for every $c>0$ ,
	\[
	\sup_{\Gamma\in\mathcal S_{\rm prd}^{1}}
	P\left(
	\sup_{t\in[0,T]}\abs{\int_{(0,t]\times\O}\Gamma\Xi\,{\rm d}\Theta}\ge c
	\right)
	\le
	\sup_{\Gamma\in\mathcal S_{\rm prd}^{1}}
	P\left(
	\abs{\int_{(0,T]\times\O}\Gamma\Xi\,{\rm d}\Theta}\ge c
	\right).
	\]
\end{lemma}

\begin{proof}
	Fix $\Gamma\in\mathcal S_{\rm prd}^{1}$. Refining partitions if necessary, write
	\[
	\Gamma\Xi
	=
	\sum_{i=0}^{n-1}\eta_i\,\1_{(t_i,t_{i+1}]}\1_{A_i},
	\]
	where $0=t_0<\cdots<t_n=T$, $\eta_i$ is $\F_{t_i}$-measurable, and disjoint sets $A_i\in\Borel_b(\O)$. Set
	$Z_i:=\eta_i\Theta((t_i,t_{i+1}]\times A_i)$ and
	$S_k:=\sum_{i=0}^{k-1}Z_i$ for $k=1,\dots,n$, with $S_0:=0$.
	For $c>0$, define
	$B_i:=\left\{\sup_{k=0,\dots,i}\abs{S_k}<c\right\}$ for $i=0,\dots,n-1$.
	Then $B_i\in\F_{t_i}$. Hence
	\[
	\widetilde\Gamma
	:=
	\sum_{i=0}^{n-1}\Gamma\,\1_{B_i}\1_{(t_i,t_{i+1}]}\1_{A_i}
	\]
	belongs to $\mathcal S_{\rm prd}^{1}$. If $\rho:=\inf\{k=1,\dots,n:\abs{S_k}\ge c\}$, then on $\{\rho<\infty\}$ we have
	$
	\sum_{i=0}^{n-1}\1_{B_i}Z_i=S_\rho.
	$
	Consequently,
	\[
	\left\{\sup_{t\in[0,T]}\abs{\int_{(0,t]\times\O}\Gamma\Xi\,{\rm d}\Theta}\ge c\right\}
	\subseteq
	\left\{\abs{\int_{(0,T]\times\O}\widetilde\Gamma\Xi\,{\rm d}\Theta}\ge c\right\}.
	\]
	Taking probabilities and  the supremum over $\Gamma\in\mathcal S_{\rm prd}^{1}$ completes the proof.
\end{proof}

\begin{remark}
	A random measure is a mapping
	$
	X\colon\Borel_b([0,T]\times\O)\to L_P^0(\Omega;\R)
	$
	which satisfies Part~(a) of Definition~\ref{de.random-measure-inf} for $\RR=\Borel_b([0,\infty)\times \O)$ and Part (a) of Definition~\ref{de.Levy-space-time}. For such a random measure $X$ and every $\Gamma\in\S_{\rm prd}^1$, the stochastic integral can be defined  by the same finite-sum
	construction as in \eqref{eq.def-int}. This yields the associated integral process
	\[
	(\Gamma\cdot X)_t
	:=
	\int_{(0,t]\times\O}\Gamma\,\d X
	=
	\int_{(0,T]\times\O}\1_{(0,t]}(s)\Gamma(s,r)\, X(\d s,\d r),
	\qquad t\in (0,T].
	\]
	For random measures $X$ and $Y$, define
	\[
	d(X,Y)
	:=
	\sup_{\Gamma\in\S_{\rm prd}^1}
	E\left[
	\sup_{t\in( 0,T]}
	\left|
	(\Gamma\cdot(X-Y))_t
	\right|
	\wedge1
	\right].
	\]
	This functional can be viewed as a space-time analogue of the \'Emery metric on the space of real-valued semimartingales introduced in \cite{Emery}. Indeed, the random measures are indexed by subsets of $[0,T]\times\O$, while the bounded predictable space-time processes in $\S_{\rm prd}^1$ play the role of bounded predictable integrands in the classical \'Emery topology.
	
	Now let $\Theta\colon\Borel_b([0,T]\times\O)\to L_P^0(\Omega;\R)$
	be a L\'evy basis and let $\Xi\in\S_{\rm prd}$. Associated with the integral $I(\Xi)$ defined in \eqref{eq.def-int} is the random measure
	\[
	I(\Xi)(B)
	:=
	\int_{(0,T]\times\O} \1_{B} \Xi\,\d\Theta,
	\qquad
	B\in\Borel_b([0,T]\times\O).
	\]
   Its existence will follow from Theorem \ref{th.space-predictable-integrands} below.
	For every $\Gamma\in\S_{\rm prd}^1$, we have
	\[
	(\Gamma\cdot I(\Xi))_t
	=
	\int_{(0,t]\times\O}\Gamma(s,r)\Xi(s,r)\,\Theta(\d s,\d r),
	\qquad t\in[0,T].
	\]
	Consequently, for all $\Xi_1,\Xi_2\in\S_{\rm prd}$,
	\[
	d\big(I(\Xi_1),I(\Xi_2)\big)
	=
	\sup_{\Gamma\in\S_{\rm prd}^1}
	E\left[
	\sup_{t\in(0,T]}
	\left|
	\int_{(0,t]\times\O}
	\Gamma \big(\Xi_1-\Xi_2\big)
	\,\d\Theta
	\right|
	\wedge1
	\right].
	\]
	Hence, by Lemma \ref{le.Emery-ucp},  the convergence required in Part~(b) of Definition~\ref{de.RM-pred_integrability} is precisely convergence with respect to the \'Emery-type metric $d$ on the space of random measures. 
\end{remark}

Our approach to stochastic integration relies on decoupling inequalities. The following result will be fundamental in what follows.
\begin{proposition} \label{pro.RM-dec_tan_seq}
	Let $\Theta\colon \Borel_b(\Rp\times \O) \to L_P^0(\Omega;\R)$ be a L\'evy basis, let
	$0=t_0<  t_1< \cdots < t_n=T$	be a partition of $[0,T]$, let $\psi_k$ be an $\F_{t_{k-1}}$-measurable real-valued random variable, and let $A_k\in\Borel_b(\O)$ for $k=1,\dots,n$.	
	Define 
	\[
	\widetilde{\Theta}\colon \Borel_b(\Rp\times \O) \to  L^0_{P \otimes P'}(\Omega \times \Omega'; \R), 
	\qquad 
	\widetilde{\Theta}(D)(\omega,\omega'):=\Theta(D)(\omega').
	\]
	Then $\widetilde{\Theta}$ is a L\'evy basis, and the sequence 
	\[
	\Big(\psi_k \widetilde{\Theta}((t_{k-1},t_k]\times A_k)\Big)_{k=1,\dots,n}
	\]
	defined on
	\[
	\big(\Omega \times \Omega',\mathcal{F} \otimes \mathcal{F}',P \otimes P',(\mathcal{F}_{t_k} \otimes \mathcal{F}'_{t_k})_{k=0,\dots,n}\big)
	\]
	is a decoupled tangent sequence to the sequence
	\[
	\Big(\psi_k \Theta((t_{k-1},t_k]\times A_k)\Big)_{k=1,\dots,n}
	\]
	defined on $\big(\Omega,\mathcal{F},P,(\mathcal{F}_{t_k})_{k=0,\dots,n}\big)$.
\end{proposition}

\begin{proof}
	Since, for every finite family $D_1,\dots,D_m\in\Borel_b(\Rp\times\O)$, the random vectors
	$(\widetilde{\Theta}(D_1),\dots,\widetilde{\Theta}(D_m))$ and $(\Theta(D_1),\dots,\Theta(D_m))$
	have the same distribution, it follows that $\widetilde{\Theta}$ is an independently scattered infinitely divisible random measure. Moreover, the defining properties of a L\'evy basis are inherited from $\Theta$ by construction.

	Now set
	\[
	\xi_k:=\Theta((t_{k-1},t_k]\times A_k),
	\qquad
	\xi_k':=\widetilde{\Theta}((t_{k-1},t_k]\times A_k),
	\qquad k=1,\dots,n.
	\]
	Since $\Theta$ is independently scattered, the sequence $(\xi_k)_{k=1,\dots,n}$ consists of independent random variables. Moreover, $(\xi_k')_{k=1,\dots,n}$ is an independent copy of $(\xi_k)_{k=1,\dots,n}$.
	Since $\psi_k$ is $\F_{t_{k-1}}$-measurable for each $k=1,\dots,n$, Example \ref{ex.decoupling-independent} shows that
	\[
	\big(\psi_k \xi_k'\big)_{k=1,\dots,n}
	=
	\Big(\psi_k \widetilde{\Theta}((t_{k-1},t_k]\times A_k)\Big)_{k=1,\dots,n}
	\]
	is a decoupled tangent sequence to
	\[
	\big(\psi_k \xi_k\big)_{k=1,\dots,n}
	=
	\Big(\psi_k \Theta((t_{k-1},t_k]\times A_k)\Big)_{k=1,\dots,n}. \qedhere 
	\]
\end{proof}

 \section{Characterisation of predictable integrands}\label{se.characterisation-integrands}

 For a given L\'evy basis $\Theta\colon \Borel_b(\Rp\times \O) \to L_P^0(\Omega;\R)$, let $(\I_\Theta,\norm{\cdot}_{\Theta})$ be the topological vector space of deterministic integrands, which is separable by Theorem~\ref{th.RM-space}. We define $L_P^0(\Omega; \I_\Theta)$ as the space of equivalence classes of Borel-measurable mappings $\Xi:\Omega \rightarrow \I_\Theta$, endowed with the $F$-norm
 \[
 \normm{\Xi}_\Theta :=E\left[\norm{\Xi}_\Theta\wedge 1\right] \quad \text{for }\Xi \in L_P^0(\Omega; \I_\Theta).
 \]
 With this $F$-norm, $L_P^0(\Omega; \I_\Theta)$ is complete. 
 
The main result of this section is the following characterisation of the space of predictable integrands which are integrable with respect to a given L\'evy basis. 
\begin{theorem}\label{th.space-predictable-integrands}
	Let $\Theta\colon \Borel_b(\Rp\times \O)\to L_P^0(\Omega;\R)$ be a L\'evy basis. A predictable process
	$\Xi\colon\Omega\times[0,T]\times\O\to\R$ is integrable with respect to $\Theta$ if and only if the mapping
	$\omega\mapsto \Xi(\omega,\cdot,\cdot)$ belongs to $L_P^0(\Omega;\I_\Theta)$.
\end{theorem}
 This section is devoted to the proof of this result. The first result we require is standard and provides a suitable approximation of predictable integrands.  
 \begin{lemma}\label{le.approximation}
 	Let $\Theta\colon \Borel_b(\R_+\times \O)\to L_P^0(\Omega;\R)$ be a L\'evy basis with control measure $\chi$. 
 	For every predictable representative $\Xi$ with $\omega\mapsto\Xi(\omega,\cdot)\in L_P^0(\Omega;\I_\Theta)$, there exists a sequence
 	$(\Xi_k)_{k\in\N}\subseteq \mathcal S_{\rm prd}$ such that
 	$\normm{\Xi_k-\Xi}_\Theta\to 0$ and $\Xi_k\to \Xi$ 
 	$P\otimes\chi$-almost everywhere on $\Omega\times[0,T]\times\O$. 
 \end{lemma}
 
 \begin{proof}
 	The $\sigma$-algebra $\P\otimes\Borel(\O)$ on $\Omega\times[0,T]\times\O$ is generated by the $\pi$-system
 	\begin{align*}
 		\mathcal C
 		&:=
 		\big\{
 		A\times (s,t]\times B :
 		A\in\F_s,\,
 		0\le s<t\le T,\,
 		B\in\Borel_b(\O)
 		\big\}\\
 		&\qquad \cup
 		\big\{
 		A\times \{0\}\times B :
 		A\in\F_0,\,
 		B\in\Borel_b(\O)
 		\big\}.
 	\end{align*}
 	Let $H$ be the class of all bounded, non-negative predictable processes
 	$Y\colon\Omega\times[0,T]\times\O\to\Rp$ for which there exists a sequence
 	$(\Xi_k)_{k\in\N}\subseteq\mathcal S_{\rm prd}$ such that
 	$0\le \Xi_k\uparrow Y $ $P\otimes\chi$-almost everywhere.
 	Then $H$ contains the indicators of all sets in $\mathcal C$, is closed under
 	positive linear combinations, and is closed under bounded increasing limits.
 	Hence, by the functional monotone class theorem, $H$ contains every bounded,
 	non-negative $\P\otimes\Borel(\O)$-measurable function.
 	
 	Let first $\Xi$ be bounded. Applying the previous paragraph to $\Xi^+$ and
 	$\Xi^-$ yields sequences $(\Xi_k^+)$ and $(\Xi_k^-)$ in $\mathcal S_{\rm prd}$
 	such that $	0\le \Xi_k^\pm\uparrow \Xi^\pm$  $P\otimes\chi$-almost everywhere.
 	Define $\Xi_k:=\Xi_k^+-\Xi_k^-$. Then $\Xi_k\to\Xi$ $P\otimes\chi$-almost everywhere and
 	$|\Xi_k-\Xi|\le 2|\Xi|$.
 	Since
 	$\alpha\mapsto\zeta_\Theta(t,x,\alpha)+\eta_\Theta(t,x,\alpha)$ is
 	continuous and non-decreasing on $\R_+$ for every $(t,x)\in [0,T]\times \O$ and $\iota_\Theta$ satisfies the $\Delta_2$-condition \eqref{eq.Delta_2}, Lebesgue's dominated convergence
 	theorem and Fubini's theorem imply  $\iota_\Theta(\Xi_k-\Xi)\to0$ $P$-almost surely.	By Theorem~\ref{th.RM-space},
 	$\norm{\Xi_k-\Xi}_\Theta\to 0$ $P$-almost surely, and hence in probability, which proves the claim for bounded $\Xi$.
 	
 	For a general predictable $\Xi\in L_P^0(\Omega;\I_\Theta)$, define
 	$\Xi^{(n)}:= \Xi\1_{\{|\Xi|\le n\}}$ for $n\in\N$.
 	Then each $\Xi^{(n)}$ is bounded and predictable,
 	$\Xi^{(n)}\to\Xi $ $P\otimes\chi$-almost everywhere,
 	and $|\Xi^{(n)}-\Xi|\le |\Xi|$.
 	As above, dominated convergence yields $\normm{\Xi^{(n)}-\Xi}_\Theta\to0$.
 	For every $n\in\N$, choose 
 	$(\Xi_{n,m})_{m\in\N}\subseteq \mathcal S_{\rm prd}$ such that
 	$\normm{\Xi_{n,m}-\Xi^{(n)}}_\Theta\to 0$   and
 	$ \Xi_{n,m}\to \Xi^{(n)}$ $P\otimes\chi$-almost everywhere.
 	A standard diagonal argument yields a sequence in
 	$\mathcal S_{\rm prd}$ converging to $\Xi$ both in
 	$\normm{\cdot}_\Theta$ and
 	$P\otimes\chi$-almost everywhere.
 \end{proof}
 
 The following result is the analogue for predictable integrands of Proposition~\ref{pro.RM-cont_of_int_op} for deterministic integrands. In the next section, we will strengthen this result and obtain a strong continuity property of the integral operator.
 \begin{proposition} \label{pro.RM-small-if-small}
 	Let $\Theta\colon \Borel_b(\Rp\times \O) \to L_P^0(\Omega;\R)$ be a L\'evy basis and $(\Xi_k)_{k\in{\mathbb N}}$ a sequence in $\mathcal{S}_{\rm prd}$. Then the following are 
 	equivalent:
 	\begin{enumerate}
 		\item[{\rm (a)}] $\displaystyle \lim_{k\to\infty}E[\norm{\Xi_k}_\Theta\wedge 1]=0$;
 		\item[{\rm (b)}] $\displaystyle \lim_{k\to\infty} \displaystyle\sup_{\Gamma \in \mathcal{S}_{{\rm prd}}^{1}}E\Bigg[\abs{\int_{(0,T]\times\O} \Gamma \Xi_k \;{\rm d}\Theta}\wedge 1 \Bigg]=0$. 
 	\end{enumerate}
 \end{proposition}
 
 \begin{proof}
 	To prove (a) $\Rightarrow$ (b), let $\epsilon>0$ be fixed. Proposition \ref{pro.RM-cont_of_int_op} enables us to choose 		$\delta>0$ such that for every $\psi\in\mathcal{S}_{\rm det}$ we have 
 	the implication:
 	\begin{align}\label{eq.RM-det-small-if-small}
 		\norm{\psi}_\Theta \le \delta \;\Rightarrow\; 
 		\sup_{\gamma \in \mathcal{S}^{1}_{\rm det}}P\left(\abs{\int_{(0,T]\times\O} \gamma \psi  \;{\rm d}\Theta}>\epsilon\right)\le 
 		\epsilon. 	
 	\end{align}
 	Choose $k_0\in {\mathbb N}$ such that the set 
 	$A_k:=\left\{\omega \in \Omega: \norm{\Xi_k(\omega)}_{\Theta}\le \delta\right\}$
 	satisfies $P(A_k)\geq 1-\epsilon$ for all $k\ge k_0$. By recalling the definition of $\widetilde{\Theta}$ and $(\Omega',\mathcal{F}',P')$ from Proposition \ref{pro.RM-dec_tan_seq}, implication 
 	\eqref{eq.RM-det-small-if-small} implies for all $\omega\in A_k$ and  $k\geq k_0$ that 
 	\[
 	\sup_{\Gamma \in \mathcal{S}_{{\rm prd}}^{1}}
 	P'\left(\omega'\in \Omega':\abs{\left(\int_{(0,T]\times\O} \Gamma(\omega)\Xi_k(\omega) \;{\rm d}\widetilde{\Theta}(\omega,\cdot)\right)(\omega^\prime)}> \epsilon \right)\leq\epsilon.\]
 	Fubini's theorem implies for all $k\geq k_0$ and $\Gamma \in \mathcal{S}_{{\rm prd}}^{1}$ that
 	\begin{align*}
 		&(P\otimes P')\left((\omega, \omega')\in \Omega\times \Omega'\colon   \abs{\left(\int_{(0,T]\times\O} \Gamma \Xi_k \;{\rm d}\widetilde{\Theta}\right)(\omega,\omega')}> \epsilon\right)\\
 		&=\int_{\Omega} P'\left(\omega'\in \Omega':  \abs{\left(\int_{(0,T]\times\O} \Gamma(\omega) \Xi_k(\omega) \;{\rm d}\widetilde{\Theta}(\omega,\cdot)\right)(\omega')}> \epsilon\right)\,P({\rm d}\omega)
 		\leq 2\epsilon
 	\end{align*}
 	As $\epsilon>0$ is arbitrary, we obtain
 	\begin{equation}\label{eq.RM-limit_sup}
 		\lim_{k \rightarrow \infty}\sup_{\Gamma \in \mathcal{S}_{\rm prd}^{1}} E_{P \otimes P'} \Bigg[\abs{\int_{(0,T]\times\O} \Gamma \Xi_k \;{\rm d}\widetilde{\Theta}} \wedge 1\Bigg] = 0.
 	\end{equation}
 	For each $k \in \mathbb{N}$ and $\Gamma \in \mathcal{S}_{{\rm prd}}^{1}$ the integrand $\Gamma \Xi_k$ lies  in ${\mathcal S}_{\rm prd}$ and has a representation of the form
 	\begin{align}\label{eq.presentiation-GammaS}
 		\Gamma\Xi_k= \sum_{\ell=0}^{n_k-1} \zeta_\ell^k\xi_\ell^k \mathbb{1}_{(t_\ell^k,t_{\ell+1}^k]} \1_{A_\ell^k},
 	\end{align}
 	where $0=t_0^k < \cdots <t_{n_k}^k= T$,   $\zeta_\ell^k\xi_\ell^k$ is an $\mathcal{F}_{t_\ell^k}$-measurable real-valued random variable, and $A_\ell^k\in \Borel_b(\O)$ for each $\ell = 0,...,n_k-1$ and $k\in\N$.
 	Proposition \ref{pro.RM-dec_tan_seq} guarantees for each $k \in \mathbb{N}$ that the sequence 
 	\[\Big(\zeta_\ell^k\xi_\ell^k \Theta\big( (t_\ell^k,t_{\ell+1}^k]\times A_\ell^k)\big)\Big)_{\ell=0,...,n_k-1}\]
 	has the decoupled tangent sequence
 	\[\Big(\zeta_\ell^k\xi_\ell^k \widetilde{\Theta}\big( (t_\ell^k,t_{\ell+1}^k]\times A_\ell^k)\big)\Big)_{\ell=0,...,n_k-1}. \]
 	We conclude from  Theorem \ref{th.decoupling} that there exists a constant $c>0$ such that, for all $k \in \mathbb{N}$ and $\Gamma \in \mathcal{S}_{{\rm prd}}^{1}$, we have
 	\begin{align*}
 		E_{P\otimes P'}\Bigg[\abs{\int_{(0,T]\times\O} \Gamma \Xi_k \;{\rm d}\Theta}\wedge 1\Bigg]
 		&= E_{P\otimes P'}\Bigg[\abs{\sum_{\ell=1}^{n_k-1} \zeta_\ell^k\xi_\ell^k \Theta((t_{\ell}^k,t_{\ell+1}^k]\times A_\ell^k) )}\wedge 1\Bigg] \nonumber\\
 		& \leq c E_{P\otimes P'}\Bigg[\abs{\sum_{\ell=1}^{n_k-1} \zeta_\ell^k\xi_\ell^k \widetilde{\Theta}((t_\ell^k,t_{\ell+1}^k]\times A_\ell^k)} \wedge 1\Bigg]\nonumber \\
 		&=c E_{P\otimes P'}\Bigg[\abs{\int_{(0,T]\times\O} \Gamma \Xi_k \;{\rm d}\widetilde{\Theta}}\wedge 1\Bigg].
 	\end{align*}
 	We conclude from  \eqref{eq.RM-limit_sup} that 
 	\begin{align*} 
 	 &	\lim_{k \rightarrow \infty}\sup_{\Gamma \in \mathcal{S}_{{\rm prd}}^{1}} E_{P}\Bigg[\abs{\int_{(0,T]\times\O} \Gamma \Xi_k \;{\rm d}\Theta}\wedge 1\Bigg]\\
 	 &\qquad\qquad = \lim_{k \rightarrow \infty}\sup_{\Gamma \in \mathcal{S}_{{\rm prd}}^{1}} E_{P\otimes P'}\Bigg[\abs{\int_{(0,T]\times\O} \Gamma \Xi_k \;{\rm d}\Theta}\wedge 1\Bigg]=0,
 	\end{align*}
 	which shows (b).

 	For establishing (b) $\Rightarrow$ (a), we assume that $\Xi_k$ is of the form 
 	\begin{align*}
 		\Xi_k=\sum_{\ell=0}^{n_k-1} \xi_\ell^k \mathbb{1}_{(t_\ell^k,t_{\ell+1}^k]}\1_{A_\ell^k},
 	\end{align*}
 	where $0=t_0^k\leq \cdots <t_{n_k}^k= T$,   $\xi_\ell^k$ is an $\mathcal{F}_{t_\ell^k}$-measurable real-valued random variable, and $A_{\ell}^k\in \Borel_b(\O)$ for $\ell=0,\dots, n_k-1$ and $k\in\N$. For every choice of signs $(\epsilon_\ell)_{\ell=0}^{n_k-1}\in\{\pm1\}^{n_k}$, the process
 	\[
 	\Gamma_\epsilon
 	:=
 	\sum_{\ell=0}^{n_k-1}\epsilon_\ell
 	\1_{(t_\ell^k,t_{\ell+1}^k]}\1_{A_\ell^k}
 	\]
 	belongs to $\mathcal S_{\rm prd}^{1}$. Hence, Theorem \ref{th.decoupling}  implies that there exists a constant $c>0$ such that
 	\begin{align}
 		E_{P \otimes P'}\Bigg[\abs{\int_{(0,T]\times\O} \Xi_k \;{\rm d}\widetilde{\Theta}}\wedge 1\Bigg] 
 		&= E_{P \otimes P'}\Bigg[\abs{\sum_{\ell=0}^{n_k-1}\xi_\ell^k\widetilde{\Theta}((t_{\ell}^k,t_{\ell+1}^k]\times A_\ell^k)}\wedge 1\Bigg] \nonumber\\
 		&\le c \max_{\epsilon_\ell \in \{\pm1\}} E_{P}\Bigg[\abs{\sum_{\ell=0}^{n_k-1}\epsilon_\ell\xi_\ell^k \Theta((t_{\ell}^k,t_{\ell+1}^k]\times A_\ell^k)}\wedge 1\Bigg] \nonumber\\
 		&\le c \sup_{\Gamma \in \mathcal{S}_{{\rm prd}}^{1}} E_{P}\Bigg[\abs{\int_{(0,T]\times\O} \Gamma \Xi_k \;{\rm d}\Theta}\wedge 1\Bigg].\label{main_proof_decoupl_ineq}
 	\end{align}
 	The hypothesis in Part (b) implies that $(\int_{(0,T]\times\O} \Xi_k\,  \d \widetilde{\Theta})_{k\in\N}$ converges to $0$ in 
 	probability. It follows that for every subsequence  of $(\Xi_k)_{k \in \mathbb{N}}$, there exists a further subsequence for which this convergence is $P\otimes P^\prime$-a.s. For ease of notation, we denote this sub-subsequence by  $(\Xi_k)_{k\in\N}$. Thus, there exists a set $N \subseteq \Omega \times \Omega'$ with $(P \otimes P')(N)=0$ satisfying
 	\[\lim_{k \rightarrow \infty}\Bigg(\int_{(0,T]\times\O} \Xi_{k}\;{\rm d}\widetilde{\Theta}\Bigg)(\omega,\omega')=0 \quad \text{for each}\; (\omega,\omega')\in N^c.\]
 	Define the section of the set $N$ for each $\omega \in \Omega$  by
 	\[N_{\omega}=\Bigg\{\omega' \in \Omega'\colon  \lim_{k\rightarrow \infty}\left(\int_{(0,T]\times\O} \Xi_{k}(\omega)\;{\rm d} \widetilde{\Theta}(\omega,\cdot)\right)(\omega')\neq 0\Bigg\},\]
 	where we note that since $\Xi_{k}$ are step processes, it holds that
 	\[\Bigg(\int_{(0,T]\times\O} \Xi_{k}\;{\rm d}\widetilde{\Theta}\Bigg)(\omega,\cdot)
 	=\int_{(0,T]\times\O} \Xi_{k}(\omega) \;{\rm d} \widetilde{\Theta}(\omega,\cdot)\qquad\text{for all }\omega\in\Omega .\]
 	Fubini's theorem implies $0=(P \otimes P')(N)= \int_{\Omega}P'(N_{\omega}){\rm d}P(\omega)$,  from which it follows that there exists  $\Omega_1 \subseteq \Omega$ with $P(\Omega_1)=1$ such that $P'(N_{\omega})=0$  for all $\omega \in \Omega_1$. In other words, for each fixed $\omega \in \Omega_1$, the  real-valued random variables
 	\[ \widetilde{I}(\Xi_{k}(\omega)):=\widetilde{I}(\Xi_{k}(\omega,\cdot,\cdot)):=\int_{(0,T]\times\O} \Xi_{k}(\omega,t,x) \; \widetilde{\Theta}(\omega,\cdot,\d t, \d x) \]
 	converges $P'$-a.s.\ to $0$ as $k\to\infty$ on $(\Omega',\mathcal{F}',P')$. Since $ \widetilde{I}(\Xi_{k}(\omega))$ is the sum of independent infinitely divisible random variables, it is also infinitely divisible. Using the notations introduced in \eqref{eq.RD-multiplied-1} and \eqref{eq.RD-multiplied-2}, its characteristics is given as 
 	\[\left(\sum_{\ell=0}^{n_k-1}  a_{D_\ell^k,\xi_\ell^k(\omega)}, \sum_{\ell=0}^{n_k-1} q_{D_\ell^k,\xi_\ell^k(\omega)},\sum_{\ell=0}^{n_k-1}   \lambda_{D_\ell^k, \xi_\ell^k(\omega)}\right),\]
 	where $D_\ell^k:=(t_\ell^k, t_{\ell+1}^k]\times A_\ell^k$.
 	Since $\widetilde{I}(\Xi_{k}(\omega))\to 0$ in $L_{P^\prime}^0(\Omega^\prime;\R)$ for all $\omega\in \Omega_1$, it follows from 
     	Remark VII.2.10 in \cite{JacodShiryaev} as in \eqref{eq.RM-aux1}	 that
 	\begin{align*}
 		\lim_{k\to\infty} 	\int_{(0,T]\times\O} \zeta_{\widetilde{\Theta}(\omega,\cdot)}(t,x, \Xi_{k}(\omega,t,x))\, \chi(\d t,\d x) 
 		&=0.
 	\end{align*}
 	Since  $\widetilde{\Theta}(\omega, \cdot)$ has the same characteristics as $\Theta$ for each $\omega \in \Omega$, we obtain for all $\omega \in \Omega_1$ that
 	\begin{align*}
 		\lim_{k\to\infty} 	\int_{(0,T]\times\O} \zeta_{\Theta}(t,x, \Xi_{k}(\omega,t,x))\, \chi(\d t,\d x) 
 		&=0.
 	\end{align*}
 	As  $P(\Omega_1)=1$ and $(\Xi_{k})_{k\in\N}$ was assumed to be a subsequence of an arbitrary subsequence, it follows, for all $\epsilon>0$, and the original sequence $(\Xi_k)_{k\in\N}$  that 
 	\begin{align}\label{eq.RM-small_k_main_proof}
 		\lim_{k \rightarrow \infty}P\left(\int_{(0,T]\times\O} \zeta_{\Theta}\big(t,x, \Xi_{k}(\cdot,t,x)\big)\, \chi(\d t, \d x)  >\epsilon\right)=0.
 	\end{align}
 	To finish the proof, it remains to show that for all $\epsilon>0$ we have
 	\begin{align}\label{eq.RM-small_l_main_proof}
 		\lim_{k \rightarrow \infty}P\left(\int_{(0,T]\times\O} \eta_{\Theta}\big(t,x, \abs{\Xi_{k}(\cdot,t,x)}\big)\, \chi(\d t, \d x)  >\epsilon\right)=0.
 	\end{align}
 	Let $\epsilon \in (0,1)$ be fixed. 
 	By \cite[Pro.~11.10]{Sato}, or equivalently by the continuity criterion for characteristics of infinitely divisible laws, there exists $\delta\in(0,\epsilon)$ such that every infinitely divisible random variable $Z$ with first characteristic $a_Z$ satisfies 
 	\[
 	P(|Z|>\sqrt{\delta})<\sqrt{\delta}
 	\quad\Longrightarrow\quad
 	|a_Z|<\epsilon .
 	\]
 	Applying this to the stochastic integral, we obtain  for all $f \in \I_\Theta$ that
 	\begin{align}\label{eq.implication_characteristic_small}
 		P\left(\abs{\int_{(0,T]\times\O} f \, \d \Theta} >\sqrt{\delta}\right)<\sqrt{\delta} \implies \abs{\int_{(0,T]\times\O}  a_{\Theta}(t,x,f(t,x))    \,\chi(\d t,\d x)}<\epsilon.
 	\end{align}
 	It follows from Equation \eqref{main_proof_decoupl_ineq} that there exists an $N \in \mathbb{N}$ such that for all $k \geq N$ we have
 	\begin{align}\label{eq.product_measure_goes_to_0}
 		\sup_{\Gamma \in \mathcal{S}_{{\rm prd}}^{1}}(P\otimes P')\left(\abs{\int_{(0,T]\times\O} \Gamma \Xi_k \, \d\widetilde{\Theta}}>\sqrt{\delta}\right)<\delta.
 	\end{align}
 	Chebyshev's inequality, Fubini's theorem and Equation \eqref{eq.product_measure_goes_to_0} imply for all $k \geq N$ and $\Gamma \in \mathcal{S}_{{\rm prd}}^{1}$ that
 	\begin{align}\label{eq.one_minus_delta_bound}
 		&P\left(\omega \in \Omega:P'\left(\omega' \in \Omega':\abs{\left(\int_{(0,T]\times\O} \Gamma(\omega) \Xi_k(\omega) \;{\rm d}\widetilde{\Theta}(\omega,\cdot)\right)(\omega')}>\sqrt{\delta}\right)<\sqrt{\delta}\right)\nonumber\\
 		&\geq 1 - \frac{1}{\sqrt{\delta}}\int_{\Omega}P'\left(\omega' \in \Omega':\abs{\left(\int_{(0,T]\times\O} \Gamma(\omega) \Xi_k(\omega) \;{\rm d}\widetilde{\Theta}(\omega,\cdot)\right)(\omega')}>\sqrt{\delta}\right)\,{\rm d}P(\omega)\nonumber\\
 		&= 1 - \frac{1}{\sqrt{\delta}}(P\otimes P')\left((\omega,\omega')\in \Omega \times \Omega':\abs{\left(\int_{(0,T]\times\O} \Gamma \Xi_k \;{\rm d}\widetilde{\Theta}\right)(\omega,\omega')}>\sqrt{\delta}\right)\nonumber\\
 		&\geq 1- \sqrt{\delta}.
 	\end{align}
 	Equations \eqref{eq.implication_characteristic_small} and \eqref{eq.one_minus_delta_bound} show for all $k \geq N$ and $\Gamma \in \mathcal{S}_{{\rm prd}}^{1}$ that
 	\begin{align*}
 		P\left(\abs{\int_{(0,T]\times\O} a_{\Theta} \big(t,x,\Gamma(t,x) \Xi_k(\cdot,t,x)\big)\, \chi(\d t,\d x)}<\epsilon\right)\geq 1-\sqrt{\delta},
 	\end{align*}
 	or equivalently, for all $k \geq N$ we have
 	\begin{align*}
 		\sup_{\Gamma \in \mathcal{S}_{{\rm prd}}^{1}}P\left(\abs{\int_{(0,T]\times\O} a_{\Theta} \big(t,x,\Gamma(t,x) \Xi_k(\cdot,t,x)\big)\, \chi(\d t,\d x)}\geq\epsilon\right)\leq \sqrt{\delta}.
 	\end{align*}
 	The above inequality, combined with an approximation argument using functions in $\mathcal{S}_{{\rm prd}}^{1}$ shows that for any predictable real-valued process $\Lambda\colon \Omega\times [0,T]\times\O\to \R$ uniformly bounded by 1 it holds that, for $k \geq N$, 
 	\begin{align}\label{eq.general_prdictable_probability}
 		P\left(\abs{\int_{(0,T]\times\O} a_{\Theta} \big(t,x,\Lambda(\cdot,t,x) \Xi_k(\cdot,t,x)\big)\, \chi(\d t,\d x)}\geq\epsilon\right)\leq \sqrt{\delta}.
 	\end{align} 
 	By applying this to $\Lambda(\omega,t,x)=\pi(t,x, \Xi_k(\omega,t,x))$, where $\pi$ is defined in the proof of Lemma \ref{le.RM-drift-equality}, it follows that 
 	\begin{align*}
 		P\left(\abs{\int_{(0,T]\times\O} \sup_{\abs{\beta}\le 1} a_{\Theta} \big(t,x,\beta \Xi_k(\cdot,t,x)\big)\, \chi(\d t,\d x)}\geq\epsilon\right)\leq \sqrt{\delta}.
 	\end{align*}
     Observing that  $\sup_{\abs{\beta}\le 1}a_\Theta(t,x,\beta \alpha)=\eta_\Theta(t,x,\abs{\alpha})$  since  $\beta\mapsto a_\Theta(t,x,\beta)$ is an odd function shows \eqref{eq.RM-small_l_main_proof}. Together with \eqref{eq.RM-small_k_main_proof}, this proves that $\iota_\Theta(\Xi_k)$ converges to zero in probability, and hence Part~(a) follows from Theorem~\ref{th.RM-space}.
 \end{proof}

 \begin{proof}[Proof of Theorem \ref{th.space-predictable-integrands}.]
 	Let $\Xi$ be a  predictable process integrable with respect to $\Theta$. Then there exists a sequence $(\Xi_k)_{k \in \mathbb{N}}$ of elements of $\mathcal{S}_{\rm prd}$ converging $P\otimes \chi$-a.e.\ to $\Xi$ and satisfying 
 	\[\displaystyle \lim_{k,\ell \rightarrow \infty}\sup_{\Gamma \in \mathcal{S}_{{\rm prd}}^{1}}E\Bigg[\abs{\int_{(0,T]\times\O} \Gamma(\Xi_k-\Xi_\ell) \;{\rm d}\Theta}\wedge1 \Bigg]=0.\]
 	Proposition \ref{pro.RM-small-if-small} implies that $\lim_{k,\ell \rightarrow \infty}\normm{\Xi_k-\Xi_\ell}_\Theta=0$. Completeness of the metric space $(L_P^0(\Omega; \I_\Theta),\normm{\cdot}_\Theta)$ and the fact that $(\Xi_k)_{k \in \mathbb{N}}$ converges $P\otimes \chi$-a.e.\ to $\Xi$ together yield that the sequence $(\Xi_k)_{k \in \mathbb{N}}$ has a limit in $L_P^0(\Omega; \I_\Theta)$ and that this limit necessarily coincides with $\Xi$. Thus $\Xi\in L_P^0(\Omega; \I_\Theta)$. 
 	
 	To establish the reverse inclusion, let $\Xi$ be a predictable process in the space $L_P^0(\Omega; \I_\Theta)$. By Lemma \ref{le.approximation}, there exists a sequence $(\Xi_k)_{k \in \mathbb{N}}$ of elements of $\mathcal{S}_{\rm prd}$ converging to $\Xi$ in $\normm{\cdot}_\Theta$ and $P\otimes\chi$-a.e. Proposition \ref{pro.RM-small-if-small} implies that 
 	\[\displaystyle \lim_{k,\ell \rightarrow \infty}\sup_{\Gamma \in \mathcal{S}_{{\rm prd}}^{1}}E\Bigg[\abs{\int_{(0,T]\times\O} \Gamma(\Xi_k-\Xi_\ell) \;{\rm d}\Theta}\wedge1 \Bigg]=0.\]
 	Thus, $\Xi$ satisfies the conditions of Definition \ref{de.RM-pred_integrability}.
 \end{proof}

\section{Limit theorems}\label{se.limit-theorems}

We begin this section by generalising Proposition \ref{pro.RM-small-if-small} from step processes to arbitrary sequences of integrable predictable processes. In this way, we obtain not only a powerful limit theorem, but also continuity of the integral operator.

 \begin{corollary} \label{co.RM-small-if-small}
 	Let $\Theta\colon \Borel_b(\Rp\times \O) \to L_P^0(\Omega;\R)$ be a L\'evy basis and $(\Xi_k)_{k\in{\mathbb N}}$ a sequence of $\Theta$-integrable predictable processes. Then the following are 
 	equivalent:
 	\begin{enumerate}
 		\item[{\rm (a)}] $\displaystyle \lim_{k\to\infty}E[\norm{\Xi_k}_\Theta\wedge 1]=0$;
 		\item[{\rm (b)}] $\displaystyle \lim_{k\to\infty} \displaystyle\sup_{\Gamma \in \mathcal{S}_{{\rm prd}}^{1}}E\Bigg[\abs{\int_{(0,T]\times \O} \Gamma \Xi_k \;{\rm d}\Theta}\wedge 1 \Bigg]=0$. 
 	\end{enumerate}
 \end{corollary}
 \begin{proof}
 	(a) $\Rightarrow$ (b): Let $\epsilon>0$ be fixed. 	
 	Proposition \ref{pro.RM-small-if-small} implies that there exists a $\delta(\epsilon)>0$ such that for all $\Xi  \in \mathcal{S}_{\rm prd}$ we have the implication:
 	\begin{equation}\label{epsilon_delta_implication}
 		\normm{\Xi}_\Theta\le \delta(\epsilon) \quad \Rightarrow \quad \displaystyle\sup_{\Gamma \in \mathcal{S}_{{\rm prd}}^{1}}E\Bigg[\abs{\int_{(0,T]\times\O} \Gamma \Xi \;{\rm d}\Theta}\wedge 1 \Bigg]\le \epsilon.
 	\end{equation}
 	Since $\lim_{k\to\infty}\normm{\Xi_k}_\Theta=0$, there exists a $k_0 \in \mathbb{N}$ such that $\normm{\Xi_k}_\Theta<\tfrac{\delta(\epsilon)}{2}$ for all $k \geq k_0$. Furthermore, Lemma \ref{le.approximation} guarantees for each $k \in \mathbb{N}$ the existence of a sequence $(\Xi_k^\ell)_{\ell \in \mathbb{N}}\subseteq \mathcal{S}_{\rm prd}$ converging to $\Xi_k$ in $\normm{\cdot}_\Theta$ and $P\otimes \chi$-a.e. Consequently, for each $k \in \mathbb{N}$ we can find $\ell_0(k,\epsilon) \in \mathbb{N}$ such that for all $\ell \geq \ell_0(k,\epsilon)$ we have 
 	\[
 	\normmm{\Xi_k^\ell-\Xi_k}_\Theta\le \tfrac{\delta(\epsilon)}{2}.
 	\]
 	Therefore, for each $k\geq k_0$ and $\ell\geq \ell_0(k,\epsilon)$,
 	\[
 	\normmm{\Xi_k^\ell}_\Theta\leq \normmm{\Xi_k^\ell-\Xi_k}_\Theta+\normm{\Xi_k}_\Theta\le  \delta(\epsilon),
 	\]
 	which implies by \eqref{epsilon_delta_implication} that
 	\begin{equation}\label{eq_small_approx_int}
 		\displaystyle\sup_{\Gamma \in \mathcal{S}_{{\rm prd}}^{1}}E\Bigg[\abs{\int_{(0,T]\times\O} \Gamma \Xi_k^\ell \;{\rm d}\Theta}\wedge 1 \Bigg]\le \epsilon.
 	\end{equation}
 	Fix $k\geq k_0$. Since the stochastic integral is independent of the approximating sequence and
 	$\Xi_k^\ell \to \Xi_k$ in $\normm{\cdot}_\Theta$ and $P\otimes\chi$-a.e., it follows from \eqref{eq_small_approx_int} that
 	\begin{align*}
 		\displaystyle\sup_{\Gamma \in \mathcal{S}_{{\rm prd}}^{1}}E\Bigg[\abs{\int_{(0,T]\times\O} \Gamma \Xi_k \;{\rm d}\Theta}\wedge 1 \Bigg]
 		&= \sup_{\Gamma \in \mathcal{S}_{{\rm prd}}^{1}} \lim_{\ell \rightarrow \infty}E\Bigg[\abs{\int_{(0,T]\times\O} \Gamma \Xi_k^\ell \;{\rm d}\Theta}\wedge 1 \Bigg] \\
 		&\leq \limsup_{\ell \rightarrow \infty} \sup_{\Gamma \in \mathcal{S}_{{\rm prd}}^{1}}E\Bigg[\abs{\int_{(0,T]\times\O}  \Gamma \Xi_k^\ell \;{\rm d}\Theta}\wedge 1 \Bigg] \\
 		&\le  \epsilon.
 	\end{align*}
 	
 	(b) $\Rightarrow$ (a): Let $\epsilon>0$ be fixed. Proposition \ref{pro.RM-small-if-small} implies that there exists a $\delta(\epsilon)>0$ such that for all $\Xi \in \mathcal{S}_{\rm prd}$ we have the implication
 	\begin{equation}\label{reverse_epsilon_delta_implication}
 		\displaystyle\sup_{\Gamma \in \mathcal{S}_{{\rm prd}}^{1}}E\Bigg[\abs{\int_{(0,T]\times\O} \Gamma \Xi \;{\rm d}\Theta}\wedge 1 \Bigg]\le \delta(\epsilon) \quad \Rightarrow \quad \normm{\Xi}_\Theta\le\tfrac{\epsilon}{2}.
 	\end{equation}
 	By assumption, there exists a $k_0 \in \mathbb{N}$ such that for all $k \geq k_0$ we have
 	\begin{equation} \label{eq.small_integral}
 		\displaystyle\sup_{\Gamma \in \mathcal{S}_{{\rm prd}}^{1}}E\Bigg[\abs{\int_{(0,T]\times\O} \Gamma \Xi_k \;{\rm d}\Theta}\wedge 1 \Bigg]\le \tfrac{\delta(\epsilon)}{4}.
 	\end{equation}
 	Lemma \ref{le.approximation} guarantees that for each $k \in \mathbb{N}$ there exists a sequence $(\Xi_k^\ell)_{\ell \in \mathbb{N}}$ in $\mathcal{S}_{\rm prd}$ converging to $\Xi_k$ in $\normm{\cdot}_\Theta$ and $P\otimes \chi$-a.e. Consequently, for each $k \in \mathbb{N}$ we can find $\ell_0(k,\epsilon) \in \mathbb{N}$ such that for all $\ell \geq \ell_0(k,\epsilon)$ we have
 	\begin{equation} \label{eq.process_dif_is_small}
 		\normmm{\Xi_k^\ell-\Xi_k}_\Theta \le \epsilon/2.
 	\end{equation}
 	Since for each $k \in \mathbb{N}$ we have $\lim_{\ell \rightarrow \infty}\normm{\Xi_k^\ell-\Xi_k}_\Theta=0$, the implication $(a)\Rightarrow(b)$ already proved above implies that for each $k \in \mathbb{N}$ there exists an $\ell_1(k,\epsilon) \in \mathbb{N}$ such that for all $\ell \geq \ell_1(k,\epsilon)$,
 	\begin{align}\label{eq.integrals_close}
 		\sup_{\Gamma \in \mathcal{S}_{{\rm prd}}^{1}}E\Bigg[\abs{\int_{(0,T]\times \O} \Gamma (\Xi_k-\Xi_k^\ell) \;{\rm d}\Theta }\wedge 1 \Bigg]\le \tfrac{\delta(\epsilon)}{2}.
 	\end{align}
 	By combining \eqref{eq.small_integral} and \eqref{eq.integrals_close}, we obtain for all $k\geq k_0$ and $\ell \geq \max\{\ell_0(k,\epsilon),\ell_1(k,\epsilon)\}$ that
 	\[
 	\sup_{\Gamma \in \mathcal{S}_{{\rm prd}}^{1}}E\Bigg[\abs{\int_{(0,T]\times\O} \Gamma \Xi_k^\ell \;{\rm d}\Theta}\wedge 1 \Bigg]\le \delta(\epsilon),
 	\]
 	which implies by \eqref{reverse_epsilon_delta_implication} and \eqref{eq.process_dif_is_small} that
 	\[
 	\normm{\Xi_k}_\Theta\leq \normmm{\Xi_k-\Xi_k^\ell}_\Theta+\normmm{\Xi_k^\ell}_\Theta\le \epsilon.
 	\]
 	As $\epsilon>0$ was arbitrary, this concludes the proof.
 \end{proof}
 
Corollary \ref{co.RM-small-if-small}, together with Lemma \ref{le.Emery-ucp}, shows that the integral operator $I_\Theta$ is continuous from $L_P^0(\Omega;\I_\Theta)$ equipped with the $F$-norm $\normm{\cdot}_\Theta$ into the space of c\`adl\`ag processes equipped with the Emery-type topology induced by the seminorm
\[
\Xi\mapsto
\sup_{\Gamma\in\mathcal S_{\rm prd}^{1}}
E\Bigg[
\sup_{t\in[0,T]}
\Big|
\int_{(0,t]\times\O}\Gamma\Xi\,{\rm d}\Theta
\Big|
\wedge1
\Bigg].
\]

We finish this section with a stochastic dominated convergence theorem.
\begin{theorem} \label{stoch_dom_conv}
	Let $\Theta\colon \Borel_b(\Rp\times \O) \to L_P^0(\Omega;\R)$ be a L\'evy basis and $(\Xi_k)_{k\in{\mathbb N}}$ a sequence of  predictable processes $\Xi_k\colon\Omega\times [0,T]\times\O\to\R$ satisfying:
	\begin{enumerate}[\rm (a)]
		\item \label{thm_cnd1'} $(\Xi_k)_{k \in \mathbb{N}}$ converges $P\otimes \chi$-a.e.\ to a predictable process $\Xi$;
		\item \label{thm_cnd2'} there exists a predictable, $\Theta$-integrable process $\Upsilon$ such that, for all $k \in \mathbb{N}$,
		\[
		(\zeta_\Theta+\eta_\Theta)(t,x,\abs{\Xi_k(\omega,t,x)})
		\leq
		(\zeta_\Theta+\eta_\Theta)(t,x,\abs{\Upsilon(\omega,t,x)}),
		\]
		for $P\otimes \chi$-a.a.\ $(\omega,t,x)\in \Omega\times [0,T]\times \O$.
	\end{enumerate}
	Then $\Xi$ is $\Theta$-integrable and
	\[
	\lim_{k\to\infty} \sup_{\Gamma \in \mathcal{S}_{{\rm prd}}^{1}}
	E\Bigg[\abs{\int_{(0,T]\times \O} \Gamma (\Xi_k-\Xi) \;{\rm d}\Theta}\wedge 1 \Bigg]=0.
	\]
\end{theorem}
\begin{proof}
	By assumption, there exists a set $N \subseteq \Omega \times [0,T]\times \O$ with $P\otimes\chi (N)=0$ such that
	\[
	\lim_{k \rightarrow \infty} \Xi_k(\omega,t,x)=\Xi(\omega,t,x)
	\]
	and
	\[
	(\zeta_\Theta+\eta_\Theta)(t,x,\abs{\Xi_k(\omega,t,x)})
	\leq
	(\zeta_\Theta+\eta_\Theta)(t,x,\abs{\Upsilon(\omega,t,x)})
	\]
	for all $(\omega,t,x)\in N^c$ and $k \in \mathbb{N}$. Fubini's theorem implies that there exists a set $\Omega_1\subseteq \Omega$ with $P(\Omega_1)=1$ such that for each $\omega \in \Omega_1$ we have
	\[
	(\zeta_\Theta+\eta_\Theta)(t,x,\abs{\Xi_k(\omega,t,x)})
	\leq
	(\zeta_\Theta+\eta_\Theta)(t,x,\abs{\Upsilon(\omega,t,x)})
	\]
	and
	\[
	\lim_{k\rightarrow \infty} \Xi_k(\omega,t,x)=\Xi(\omega,t,x)
	\]
	for $\chi$-a.a.\ $(t,x) \in [0,T]\times \O$. Monotonicity of $\zeta_\Theta+\eta_\Theta$ and the $\Delta_2$-condition \eqref{eq.Delta_2} yield
	\[
	(\zeta_\Theta+\eta_\Theta)(t,x,\abs{\Xi_k(\omega,t,x)-\Xi_\ell(\omega,t,x)})
	\leq
	5\,(\zeta_\Theta+\eta_\Theta)(t,x,\abs{\Upsilon(\omega,t,x)}).
	\]
	Furthermore, Theorem \ref{th.space-predictable-integrands} guarantees that there exists a set $\Omega_2 \subseteq \Omega$ with $P(\Omega_2)=1$ such that
	\[
	\iota_\Theta(\Upsilon(\omega,\cdot))<\infty
	\qquad \text{for all $\omega \in \Omega_2$.}
	\]
	Since $\zeta_\Theta$ and $\eta_\Theta$ are continuous, another application of Lebesgue's dominated convergence theorem implies that for all $\omega \in \Omega_1 \cap \Omega_2$,
	\begin{align*}
		&\lim_{k,\ell \rightarrow \infty}\iota_\Theta(\Xi_k(\omega,\cdot)-\Xi_\ell(\omega,\cdot))\\
		&\qquad = \lim_{k,\ell \rightarrow \infty}\Bigg( \int_{(0,T]\times\O} \Big(\zeta_\Theta(t,x,\abs{\Xi_k(\omega,t,x)-\Xi_\ell(\omega,t,x)}) \Big. \Bigg.\\
		&\Bigg.\Big.\qquad\qquad\qquad\qquad\qquad\quad
		+\eta_\Theta(t,x,\abs{\Xi_k(\omega,t,x)-\Xi_\ell(\omega,t,x)})\Big)\,\chi(\d t,\d x)\Bigg)\\
		&\qquad =0.
	\end{align*}
	Hence, for each $\omega \in \Omega_1 \cap \Omega_2$, the sequence $(\Xi_k(\omega))_{k \in \mathbb{N}}$ is Cauchy in $(\I_\Theta,\norm{\cdot}_{\Theta})$. Since $\Xi_k(\omega)\rightarrow \Xi(\omega)$ for $\chi$-a.a.\ $(t,x) \in [0,T]\times\O$, it follows that $\Xi(\omega)\in \I_\Theta$. As $P(\Omega_1 \cap \Omega_2)=1$, Theorem \ref{th.space-predictable-integrands} implies that $\Xi$ is $\Theta$-integrable. Another application of Lebesgue's dominated convergence theorem yields $\lim_{k \rightarrow \infty}\normm{\Xi_k-\Xi}_\Theta=0$,	and the assertion now follows from Corollary \ref{co.RM-small-if-small}.
\end{proof}

{\bf Acknowledgement.} I thank Gergely Bodó for valuable comments on an earlier version of this paper.

\end{document}